\documentclass[11pt]{article}
\usepackage{amsmath, amsthm, amssymb}
\usepackage{enumerate}
\usepackage{mathrsfs}
\usepackage{color}

\def\XXint#1#2#3{{\setbox0=\hbox{$#1{#2#3}{\int}$}
    \vcenter{\hbox{$#2#3$}}\kern-.5\wd0}}

\def\({\left(}
\def\){\right)}

\def\RR{{\mathbb{R}}}
\def\CC{{\mathbb{C}}}

\newcommand{\weak}{\rightharpoonup}

\newcommand{\dist}{{\rm dist\, }}

\newcommand{\NN}{{\bf N}}

\newcommand{\eps}{\epsilon}

\renewcommand{\Re}{{\rm Re\,}}

\newcommand{\be}{\begin{equation}}
\newcommand{\ee}{\end{equation}}
\newcommand{\bea}{\begin{eqnarray}}
\newcommand{\eea}{\end{eqnarray}}
\newcommand{\beann}{\begin{eqnarray*}}
\newcommand{\eeann}{\end{eqnarray*}}
\newcommand{\nnn}{\nonumber}

 \newtheorem{theorem}{Theorem}[section]
\newtheorem{corollary}[theorem]{Corollary}
\newtheorem{remark}{Remark}[section]

\newtheorem{lemma}[theorem]{Lemma}
\newtheorem{claim}{Claim}
\newtheorem{prop}[theorem]{Proposition}

\newenvironment{proof1}{
    \noindent {\em Proof }}{\hfill$\Box$}

\begin{document}
\title{\bf Domain walls in the coupled Gross-Pitaevskii equations}
\author{Stan Alama, Lia Bronsard, Andres Contreras, and Dmitry E. Pelinovsky \\
{\em Department of Mathematics and Statistics, McMaster University} \\
{\em Hamilton ON, Canada, L8S 4K1}}

\thispagestyle{empty}
\maketitle

\begin{abstract}

A thorough study of domain wall solutions in coupled Gross-Pitaevskii equations on the real line
is carried out including existence of these solutions; their spectral and nonlinear stability;
their persistence and stability under a small localized potential. The proof of existence
is variational and is presented in a general framework: we show that the domain wall solutions are energy minimizing within a class of vector-valued functions with  nontrivial conditions at infinity.  The admissible energy functionals include those corresponding to coupled
Gross--Pitaevskii equations, arising in modeling of Bose-Einstein condensates. The results on spectral and nonlinear stability follow from properties
of the linearized operator about the domain wall. The methods apply to many systems of interest
and integrability is not germane to our analysis. Finally, sufficient conditions for
persistence and stability of domain wall solutions are obtained to show that stable pinning
occurs near maxima of the potential, thus giving rigorous justification to earlier results in the physics literature.
\end{abstract}



\section{Introduction}

Domain walls are ubiquitous in physical systems. The purpose of the present work is to initiate a rigorous analysis of this phenomenon by placing it in a general variational framework. We are interested in existence
and stability of these domain walls as well as in their dynamical properties.  The study is fairly complete and covers general existence and asymptotic properties of the solutions, spectral and orbital stability, and it also includes the case of a small localized potential where the spectral stability of these solutions is completely characterized.
Our perspective is mainly variational, also including some perturbation analysis in the case
of small localized potentials.

Consider the system of coupled Gross--Pitaevskii equations,
\be\label{GP}
\left. \begin{aligned}
i\partial_t \psi_1 = -\partial_x^2 \psi_1 + (g_{11}|\psi_1|^2 + g_{12}|\psi_2|^2)\psi_1, \\
i\partial_t \psi_2 = -\partial_x^2 \psi_2 + (g_{12}|\psi_1|^2 + g_{22}|\psi_2|^2)\psi_2,
\end{aligned}\right\}
\ee
where the cross-interaction terms are taken to be equal to preserve the Hamiltonian structure of the
underlying equations. The system  \eqref{GP} may be seen as the simplest model for domain walls in the real line.

Domain walls occur in many physical experiments, such as convection in fluid
dynamics \cite{M1,M2} and polarization modulation instability in fiber optics \cite{HS1,HS2}.
Recently, domain wall solutions were discussed in the coupled Bose--Einstein condensates,
both in one and two dimensions \cite{Dror}, and this is the prime motivation for our study, as the one-dimensional domain walls should represent the  leading order term in an expansion of the
energy of a two-component Bose-Einstein condensate.

 For simplicity, start with the model case $g_{11}=1=g_{22}$ and $\gamma=g_{12}$.
Stationary solutions of the form $\psi_1=e^{-it\mu_1}u_1$, $\psi_2=e^{-it\mu_2}u_2$
with real-valued envelopes $u_1$, $u_2$ and normalization $\mu_1=1=\mu_2$
satisfy the system of differential equations
\begin{equation}\label{hc}
\left.\begin{aligned}
- u''_1(x) + \left( u_1^2 + \gamma u_2^2 -1\right)u_1 &=0, \\
-u''_2(x) + \left( \gamma u_1^2 + u_2^2 -1\right)u_2 & =0.
\end{aligned}\right\} \quad x\in\RR.
\end{equation}

We seek  nonnegative solutions of system \eqref{hc},
with heteroclinic boundary conditions at infinity,
\begin{gather}
\label{minus}
u_1(x)\to 0, \quad u_2(x)\to 1, \qquad\text{as $x\to-\infty$,}\\
\label{plus}
u_1(x)\to 1, \quad u_2(x)\to 0, \qquad\text{as $x\to +\infty$.}
\end{gather}
In the special case $\gamma = 3$, such solutions are known explicitly \cite{M1,Dror}:
\begin{equation}
\label{exact}
\gamma = 3: \quad u_{1,2}(x) = \frac{1}{2} \left[ 1 \pm \tanh\left(\frac{x}{\sqrt{2}}\right) \right].
\end{equation}
Apart from the integrable cases $\gamma=1,3$, generally there is no explicit formula for the domain wall solutions. However, we will prove that such solutions exist for any $\gamma > 1$, and in fact for a large class of systems of two Hamiltonian PDE.  In addition, we will prove that
they are spectrally and nonlinearly stable. Lastly, we will add
a small localized potential to the coupled Gross--Pitaevskii equations \eqref{GP}
and prove the early observation in \cite{Dror} that the domain walls persist
near the nondegenerate extremum points of the small potentials and become
spectrally stable (unstable) near the maximum (minimum) points, thus providing a very complete picture of this phenomenon.

For the main result, our technique is variational. Therefore,
we introduce the general energy functional, for functions $\psi_j: \RR\to \CC$, $j=1,2$,
\begin{equation}
\label{energy}
E(\psi_1,\psi_2) = \frac12\int_{-\infty}^{\infty}
\left( |\nabla\psi_1|^2 + |\nabla\psi_2|^2 + W(\psi_1,\psi_2)\right) dx,
\end{equation}
where the appropriate choice of potential $W$ corresponding to equations \eqref{hc} is:
\be\label{W}
W(\psi_1,\psi_2)=\frac12 (|\psi_1|^2+|\psi_2|^2 -1)^2 + (\gamma-1)|\psi_1|^2|\psi_2|^2.
\ee
Denote by $\Psi=(\psi_1,\psi_2)$, and let $\RR^2_+$ be the set of vectors in $\RR^2$ with nonnegative coordinates: $(x,y)\in\RR^2_+$ if $x,y\ge 0$.
The potential $W$ then satisfies the following general properties:
\begin{enumerate}
\item[(W1)]  $W(\Psi)=W(|\psi_1|,|\psi_2|)=F(|\psi_1|^2,|\psi_2|^2)$ for $F\in C^3(\RR^2;\RR)$.
\item[(W2)]  $W(\Psi)\ge 0$ for all $\Psi\in\CC^2$, and there exist  $a,b>0$,  so that $W(\Psi)=0$ if and only if
$(|\psi_1|,|\psi_2|)=\mathbf a=(a,0)$ or $\mathbf b=(0,b)$.
\item[(W3)]   $\mathbf a$, $\mathbf b$ are non-degenerate global minima of $W$ (when restricted to $\RR^2$.)
\item[(W4)]  There exist constants $R_0,c_0>0$ such that
$$
\nabla W(U)\cdot U\ge  c_0  |U|^2
\quad \mbox{\rm for all $U\in\RR^2_+$ with} \quad |U|\ge R_0.
$$
\end{enumerate}

We will show that the above properties are sufficient for the existence 
of domain wall solutions, and also nearly sufficient for many of their properties, 
including dynamical stability.  A great variety of coupled equations of 
nonlinear Schr\"odinger type fit the above framework.  For instance, 
taking the general form of the coupled Gross-Pitaevskii equations \eqref{GP} 
with arbitrary $g_{11},g_{22}>0$ and   $g_{12}>\sqrt{g_{11}g_{22}}$, 
we may seek stationary domain wall solutions of the form 
$\psi_1=e^{-it\mu\sqrt{g_{11}}}u_1(x)$ and $\psi_2=e^{-it\mu\sqrt{g_{22}}}u_2(x)$, 
with $\mu>0$ any constant.  The resulting system for $U=(u_1,u_2)$ takes the form
\be\label{gen}\left.
\begin{gathered}
-u''_1 + g_{11}\left(u_1^2-a^2\right)u_1 + g_{12} u_1u_2^2=0 \\
-u''_2 + g_{22}\left(u_2^2-b^2\right)u_2 + g_{12} u_1^2u_2=0
\end{gathered}
\right\}
\ee
with
\be\label{genlim}
U(x)\to \mathbf{a}:= (a,0), \quad \text{as $x\to\infty$,}\qquad
U(x)\to \mathbf{b}:= (0,b), \quad \text{as $x\to-\infty$,}
\ee
where
$$  a= {\sqrt{\mu}\over \sqrt[4]{g_{11}}}, \qquad b= {\sqrt{\mu}\over \sqrt[4]{g_{22}}}.
$$
For this more general domain wall system, the corresponding potential is
\be\label{genpot}
W(\psi_1,\psi_2)= \frac12 (\sqrt{g_{11}}|\psi_1|^2 + \sqrt{g_{22}}|\psi_2|^2 -\mu)^2 + \left(g_{12}-\sqrt{g_{11}g_{22}}\right) |\psi_1|^2|\psi_2|^2,
\ee
which also satisfies the conditions (W1)--(W4) above, provided
$$g_{12}>\sqrt{g_{11}g_{22}}$$  
a hypothesis which we make throughout the paper.

Many other coupled Schr\"odinger systems with Hamiltonian structure 
may be treated by choosing different potentials satisfying (W1)--(W4).  
For instance,
\begin{equation}\label{ho}
W(\Psi)=\frac14(|\psi_1|^4 + |\psi_2|^4-1)^2 + \frac{\gamma-1}{2}|\psi_1|^4 |\psi_2|^4,
\end{equation}
with $\gamma>1$, is another admissible energy functional, 
which generates the system of coupled Gross--Pitaevskii equations,
$$  
\left. \begin{aligned}
i\partial_t \psi_1 = -\partial_x^2 \psi_1 + (|\psi_1|^4 + \gamma |\psi_2|^4 - 1) |\psi_1|^2 \psi_1, \\
i\partial_t \psi_2 = -\partial_x^2 \psi_2 + (\gamma |\psi_1|^4 + |\psi_2|^4 - 1) |\psi_2|^2 \psi_2,
\end{aligned}\right\}
$$
with the domain wall solutions satisfying asymptotic conditions 
$\Psi(x)\to(0,1)$ as $x\to-\infty$ and $\Psi(x)\to(1,0)$ as $x\to\infty$.  
Thus, we may consider systems other than the standard cubic Gross-Pitaevskii equations \eqref{GP}.

To find solutions with the desired conditions \eqref{genlim} at infinity,
we first start with a very weak topology.  Call  $X$ the class of all
$U=(u_1,u_2)\in H^1_{loc}(\RR;\RR^2)$ which satisfy the asymptotic conditions \eqref{genlim}.
Define also $Y$ to be the class of complex-valued
$\Psi=(\psi_1,\psi_2)\in H^1_{loc}(\RR;\CC^2)$ satisfying
$U := (|\psi_1(x)|,|\psi_2(x)|)\to\mathbf a$ as $x\to\infty$ and $U \to\mathbf b$ as $x\to -\infty$.
Although neither space is closed under $H^1_{loc}$ convergence, we will nevertheless obtain
convergence in the stronger topology defined by the family of distances, (see \cite{Bethuel})
\begin{equation}\label{distance}
\rho_A(\Psi,\Phi):= \sum_{j=1,2} \left[ \bigl\| \psi'_j - \varphi'_j \bigr\|_{L^2(\RR)}
  + \bigl\| |\psi_j|-|\varphi_j| \bigr\|_{L^2(\RR)} +   \bigl\| \psi_j -\varphi_j\bigr\|_{L^\infty(-A,A)}\right],
\end{equation}
where $A>0$ is a fixed constant. Our main existence result is given by the following theorem.

\begin{theorem}
\label{intro}
Assume $W$ satisfies (W1)--(W4).  Define
$$ m=\inf_{\Psi\in Y}  E(\Psi).
$$
Then there exists $\Psi=(\psi_1,\psi_2)\in Y$ which attains the infimum of $E$ in $Y$.
Moreover, every minimizer has the form $\psi_1=e^{i\beta_1}u_1$,
$\psi_2=e^{i\beta_2}u_2$, for nonnegative real-valued $U=(u_1,u_2)\in X$ and $\beta_1,\beta_2\in \RR$ constants.  Furthermore,
for any minimizing sequence $\Psi_n\in Y$, $E(\Psi_n)\to m$, there exists a minimizer $\Psi\in Y$, a sequence $\tau_n\in\RR$, and a subsequence for which
$$\rho_A(\Psi_{n_k}(\cdot+\tau_{n_k}),\Psi(\cdot))\to 0, \mbox{ as } k\to \infty$$  holds for all constants $A>0 .$
\end{theorem}

A more detailed theorem, giving essential properties of the minimizing domain wall solutions, 
is presented in Section 2;  see Theorem~\ref{existence}.  In particular, the solutions have 
exponential convergence as $x\to\pm\infty$ to their asymptotic limits, and in the symmetric case \eqref{hc}, 
all minimizers satisfy $0 \leq u_1(x), u_2(x) \leq 1$ and 
there exists a minimizer which is symmetric about $x=0$, $u_2(x)=u_1(-x)$.

\begin{remark}
{\rm In the symmetric case \eqref{hc}, there exists another equilibrium state
$$
{\bf c} = \left(\frac{1}{\sqrt{1+\gamma}},\frac{1}{\sqrt{1+\gamma}}\right)
$$
inside the range, where the domain wall solutions are defined. The equilibrium state ${\bf c}$ corresponds
to the center-saddle point of the dynamical system (\ref{hc}). It was reported with the numerical shooting
method in \cite{HS1} that the domain wall solutions with the symmetry $u_2(x) = u_1(-x)$
satisfy $u_2(0) = u_1(0) = \frac{1}{\sqrt{1+\gamma}}$ for any $\gamma > 1$. It remains open in the variational theory
to prove this result.}
\end{remark}

\begin{remark}
{\rm We do not know if the minimizer found in Theorem~\ref{existence} is unique. However, in Section 5, 
see Proposition~\ref{isolated}, we prove that under some general hypotheses 
the set of all energy-minimizing domain walls is discrete.}
\end{remark}

The existence of heteroclinics connecting the wells of a bistable potential $W$ have been proven by many authors.
Sternberg \cite{Stern} gave an existence proof by characterizing the heteroclinics as geodesics in a degenerate
metric, a point of view which we adopt in proving Theorem~\ref{intro}.  Connecting orbits for symmetric
potentials were found in the case of multiple-well potentials by Bronsard, Gui, \& Schatzman \cite{BGS} and
Alama, Bronsard, \& Gui \cite{ABGui}. A more general existence theorem, in the absence of symmetry hypotheses,
was found by Alikakos \& Fusco \cite{AF}, by employing constraints.  For the stability (linear and nonlinear)
we require the stronger convergence in the distance $\rho_A$ of unconstrained minimizing sequences, and thus our result is an
improvement on previous work for the two-well case.

\medskip

The existence of domain wall solutions having been established, we turn to the questions of linear and nonlinear stability, with respect to the Hamiltonian flow,
\be\label{Ham}
i\partial_t\psi_1 = -\partial_x^2\psi_1 + \partial_1 F(|\psi_1|^2,|\psi_2|^2)\psi_1, \qquad
i\partial_t\psi_2 = -\partial_x^2\psi_2 + \partial_2 F(|\psi_1|^2,|\psi_2|^2)\psi_2.
\ee
Linear and nonlinear stability of the domain wall solutions $U$ will also be proven under rather general hypotheses, nevertheless slightly more restrictive than were necessary for their existence.  The first step is to consider the linearization about $U$, $D^2 E(U)$, which is defined in $H^1_0(\mathbb{R};\mathbb{C}^2).$
Let us consider any admissible $\Phi=\Phi_R + i\Phi_I$ and  let us express $\Phi_R=(\varphi_{1,R},\varphi_{2,R})$ and
$\Phi_I=(\varphi_{1,I},\varphi_{2,I})$ in their real and imaginary parts.
We associate to the quadratic form $D^2 E(U)$ two self-adjoint linearizations, which decompose the second variation as 
$$
D^2E(U)\Phi = L_+\Phi_R + i L_-\Phi_I, 
$$
with self-adjoint operators
$$
L_+\Phi_R=\Phi''_R + \frac12 D^2W(U)\Phi_R, \qquad
L_-\Phi_I=\Phi''_I + DF(u_1^2,u_2^2):\Phi_I,
$$
where we denote $v:w=(v_1w_1, v_2w_2)$ for $v,w\in\RR^2$.
In Section 3, see Theorem~\ref{secondvar}, we prove that both $L_\pm$ 
are positive semi-definite.  In addition, $L_+$ has a zero eigenvalue, 
corresponding to the eigenfunction $U'(x)$, but the essential spectrum is bounded away from zero.

An important issue is the simplicity of the zero eigenvalue of $L_+$, 
which is sensitive to the form of the potential $W$.  Indeed, if we 
choose $g_{12}=0$ in \eqref{gen}, the equations decouple and the zero 
eigenvalue of $L_+$ will have multiplicity two.  As part of Theorem~\ref{secondvar}, 
we give a sufficient condition on $W(\Psi)=F(|\psi_1|^2,|\psi_2|^2)$ for which zero is a simple eigenvalue,
\begin{enumerate}
\item[(W5)]  $\partial_1\partial_2 F(\xi_1,\xi_2)\gneq 0$ for all $\xi=(\xi_1,\xi_2)\in\RR^2_+$,
\end{enumerate}
a condition which is satisfied by the examples \eqref{hc}, \eqref{gen}, and \eqref{ho} given above.  
For such $W$, we may also conclude the strict monotonicity of the profiles $U(x)=(u_1(x),u_2(x))$.
From this spectral analysis we also obtain a spectral stability result in the spirit of 
the work of Di Menza and Gallo \cite{Gallo} on the black soliton for the NLS equation.

\begin{theorem}
\label{theorem-stability}  If $U\in X$ is a minimizer of $E$, then the associated spectral problem
\be\label{eigenvalue}
L_+\Phi_R = -\lambda\Phi_I, \qquad L_-\Phi_I =\lambda \Phi_R
\ee
has no eigenvalues $\lambda$ with ${\rm Re}(\lambda) \neq 0$.
\end{theorem}

We note that Theorem \ref{theorem-stability} holds even if zero is a multiple (semi-simple) eigenvalue of $L_+$.  In that case, it is unnatural to claim that the system is spectrally stable, as the presence of a null vector which is not accounted for by symmetries (translation invariance, in our case) usually signals a bifurcation of stationary solutions.  In this case, perturbations from the domain wall may grow algebraically in time and no linear or nonlinear stability of the solutions may be established.  Therefore, to establish stability of domain wall solutions we restrict our attention to the case where zero is a simple eigenvalue of $L_+$, which is ensured by the hypothesis (W5).

Nonlinear stability of non-degenerate domain wall solutions can be thought to be 
a natural consequence of the minimizing character of these solutions in the energy functional $E$. 
We note, however, that the fact that domain walls have
nontrivial boundary conditions at infinity presents additional challenges, as the dynamics
in this situation is not ruled by scattering. While we have existence of the domain wall solutions, 
there is no uniqueness result, and no explicit formula for the solutions. The combination
of these two features makes the problem quite subtle.  Having that in mind, define the energy space,
\begin{equation}\label{D}
\mathcal{D}:=\{\Psi\in H^1_{loc}(\RR;\RR^2) : \ E(\Psi)<\infty\}.
\end{equation}
From \cite{Zhid}, we have the following global well-posedness result in the energy space $\mathcal{D}$.

\begin{theorem}[Zhidkow]\label{Zhidkow}
Let $\Psi_0 (x)\in \mathcal{D}\cap L^{\infty}(\mathbb{R})$.
There exists a unique global in time solution $\Psi (x,t)$ to the system \eqref{Ham} with initial data
$\Psi(x,0)=\Psi_0 (x).$ Moreover, the map $t\to \Psi(\cdot, t)$ is continuous
with respect to $\rho_A$ and  energy is preserved along the flow, that is $E(\Psi (x,t))=E(\Psi_0)$ for all $t.$
\end{theorem}

We may now state our result on the orbital stability of the domain wall solutions.  Again, the result is the same for any Gross-Pitaevskii system, as long as the associated potential $W$ satisfies (W1)--(W4).

\begin{theorem}\label{orb-stab}  Assume (W1)--(W4), let $U$ be a minimizer of $E$ in $X,$ for which zero is a simple eigenvalue of $L_+$.
 Let $\Psi_0\in \mathcal{D}\cap L^{\infty}$ and let $\varepsilon>0 .$ There exist a positive number $\delta>0$ and real functions $\alpha(t), \theta_1(t),\theta_2(t)$ such that if
\begin{equation}\label{delta-nonlinear} \rho_A (\Psi_0, U)\leq \delta, \end{equation} then

 \begin{equation}\label{epsilon-nonlinear}
  \sup_{t\in \mathbb{R}} \rho_A \left(\Psi (\cdot,t),
  \begin{bmatrix} \exp [i \theta_1 (t)] u_1(\cdot+\alpha(t) \\
     \exp [i\theta_2 (t)] u_2(\cdot +\alpha(t)
     \end{bmatrix}
  \right)\leq \varepsilon.\end{equation}
\end{theorem}

The proof of Theorem \ref{orb-stab} makes use of the variational structure of the equation and the Concentration-Compactness argument employed in proving Theorem~\ref{intro}.
In this way it recalls the classical work of Cazenave-Lions \cite{CLi} and Grillakis-Shatah-Strauss \cite{GST}. In our case, however, the control of the
phase is a very delicate matter and falls outside the Grillakis-Shatah-Strauss formalism.
In \cite{Bethuel}, orbital stability of the black soliton for the Gross-Pitaevskii  equation
was obtained facing similar problems as ours. The black soliton is a constrained minimizer
among functions with fixed `untwisted momentum'  equal to $\pi/2$ (c.f. \cite{Bethuel}) and
important part of the analysis goes into defining this notion rigorously. We do not deal
with such an issue, however, a key point needed in the analysis in \cite{Bethuel} is that
travelling waves with speed $c$ (including the case $c=0$ corrresponding to the black soliton)
are known explicitly and are unique. This complete characterization is not available to us.
Nevertheless, we are able to circumvent this by making use of the asymptotic behavior of
the domain wall solutions at $\pm \infty$ and the fact that heteroclinic minimizers are
isolated (as in Proposition \ref{isolated} below).

Note also that a stronger version of stability, namely asymptotic stability, is expected to hold
for the domain wall solutions and we hope to tackle this problem in a future project.

Even though variational techniques do not give  much information about $\theta_1(t), \theta_2(t)$
and $\alpha(t),$ a direct application of the same reasoning behind Theorem 1.3 in \cite{Bethuel}
to our setting allows us to obtain a weak form of a slow motion law for the center $\alpha(t)$ of the domain wall, at least for the family of solutions of the general Gross-Pitaevskii system \eqref{gen}.

\begin{theorem}\label{Slowmotion}
Let $U(x)\in X$ be an energy minimizing domain wall solution of \eqref{gen} with asymptotic conditions \eqref{genlim}.
Let $\alpha(t),\theta_1(t), \theta_2(t)$ be functions satisfying the conclusion of Theorem \ref{orb-stab}.
Then, there exists a constant $C=C(A)$ such that for all $t\in\mathbb{R}:$
$$
\vert\alpha(t)\vert\leq C\varepsilon\max\{1,\vert t\vert\},
$$
provided $\varepsilon$ is sufficiently small.
\end{theorem}

Finally, in Section 6, we study the influence of a small localized potential 
on the domain walls. For simplicity we treat perturbations of the model system \eqref{hc}, but the same procedure may be adapted to the more general cases.  Consider

\be\label{GP-potential}
\left. \begin{aligned}
i\partial_t \psi_1 = -\partial_x^2 \psi_1 + \epsilon V(x) \psi_1 + (|\psi_1|^2 + \gamma |\psi_2|^2)\psi_1, \\
i\partial_t \psi_2 = -\partial_x^2 \psi_2 + \epsilon V(x) \psi_2 +(\gamma |\psi_1|^2 + |\psi_2|^2)\psi_2,
\end{aligned}\right\}
\ee
where $\epsilon > 0$ is a small parameter and $V : \RR \to \RR$ is a given potential.
Stationary solutions of the form $\psi_1=e^{-it}u_1$, $\psi_2=e^{-it}u_2$
with real-valued envelopes $u_1$, $u_2$ satisfy the system of differential equations
\begin{equation}\label{hc-potential}
\left.\begin{aligned}
- u''_1(x) + \left( \epsilon V(x) + u_1^2 + \gamma u_2^2 -1\right)u_1 &=0, \\
-u''_2(x) + \left( \epsilon V(x) + \gamma u_1^2 + u_2^2 -1\right)u_2 & =0.
\end{aligned}\right\} \quad \quad x\in\RR.
\end{equation}

For $\epsilon = 0$, existence of the domain wall solutions of the system (\ref{hc-potential})
with the boundary conditions (\ref{minus}) and (\ref{plus}) is given by Theorem \ref{existence}.
By using the method of Lyapunov--Schmidt reductions, similarly to the work of
Pelinovsky \& Kevrekidis \cite{PK} on black solitons for the Gross--Pitaevskii equation with a small localized potential,
we show persistence of the domain wall solutions for small values of $\epsilon$.

\begin{theorem}
\label{theorem-persistence}
Let $U_0 = (u_1,u_2)$ be a heteroclinic solution of the system (\ref{hc}) with $\gamma > 1$
in function space $X$ satisfying the symmetry reduction $u_2(x) = u_1(-x)$ for all $x \in \mathbb{R}$.
For a given $V \in C^2(\RR) \cap L^2(\mathbb{R})$, assume that there exists $x_0 \in \RR$
such that
\begin{equation}
\label{condition-1}
\int_{\mathbb{R}} V'(x + x_0) (u_1^2 + u_2^2 - 1) dx = 0,
\end{equation}
and
\begin{equation}
\label{condition-2}
\int_{\mathbb{R}} V''(x + x_0) (u_1^2 + u_2^2 -1) dx \neq 0.
\end{equation}
Then, there exists $\epsilon_0 > 0$ such that for all $\epsilon \in (-\epsilon_0,\epsilon_0)$,
the system of differential equations (\ref{hc-potential}) admits a unique branch of the heteroclinic solutions
$U = (u_1,u_2)$ in function space $X$. Moreover, $U$ is $C^{\infty}$ in $\epsilon$ and there is $C > 0$ such that
\begin{equation}
\label{smallnest}
\sup_{x \in \RR} \left| U(x) - U_0(x-x_0) \right| \leq C |\epsilon|, \quad \epsilon \in (-\epsilon_0,\epsilon_0).
\end{equation}
\end{theorem}

\begin{remark}
{\rm In the particular case of even $V$, the first condition (\ref{condition-1}) is satisfied for $x_0 = 0$ because
$u_1^2 + u_2^2 - 1$ is even and $V'$ is odd. In this case, solutions of system (\ref{hc-potential}) for small
$\epsilon \in (-\epsilon_0,\epsilon_0)$ satisfies the symmetry reduction
$$
u_2(x) = u_1(-x) \quad \mbox{\rm for all} \quad x \in \mathbb{R},
$$
hence the bifurcation equation of the Lyapunov--Schmidt reduction (see equation (\ref{bifurcation-eq}) below)
is satisfied identically for $s = x_0 = 0$. As a result, the second condition (\ref{condition-2}) can be dropped
and it is sufficient to require $V \in C^1(\RR) \cap L^2(\RR)$ in the statement of Theorem \ref{theorem-persistence}.}
\end{remark}

Note that the effective potential, which produces the conditions (\ref{condition-1}) and (\ref{condition-2}) was
introduced in equation (48) of Ref. \cite{Dror} from physical arguments.

Once a unique branch of domain walls is shown to exist for small enough $\varepsilon,$
we turn to the stability conditions for the persistent domain wall solutions in the small
localized potential.

\begin{theorem}
\label{theorem-stability-potential}
Assume conditions of Theorem \ref{theorem-persistence} and that $V\in L^1(\mathbb{R})$.
The domain wall solutions of Theorem \ref{theorem-persistence} are spectrally stable
if $\sigma > 0$ and unstable if $\sigma < 0$, where
$$
\sigma := \frac{1}{2} \int_{\mathbb{R}} V''(x + x_0) (u_1^2 + u_2^2 -1) dx \neq 0.
$$
\end{theorem}

Note that if $u_1^2 + u_2^2 - 1 \leq 0$ for all $x \in \mathbb{R}$ and $V$ is slowly varying
on the scale of the domain wall $U = (u_1,u_2)$, then $\sigma > 0$ if $x_0$ is the point of maximum of
$V$ with $V''(x_0) < 0$. This corresponds to the prediction of Ref. \cite{Dror} based on
physical arguments that the stable pinning of the
domain walls happens at the potential maxima (rather than minima).

Let us give an example of the domain wall solution (\ref{exact}) for $\gamma = 3$ and
the explicit potential $V(x) = a \;{\rm sech}^2(bx)$ with $a \in \mathbb{R}$ and $b > 0$.
In this case, the condition (\ref{condition-1}) is satisfied for $x_0 = 0$
and the stability condition $\sigma > 0$ is satisfied if $a > 0$, that is, when $V$ is a single-humped
potential. The instability condition $\sigma < 0$ is satisfied if $a < 0$, that is, when $V$ is a
single-well potential. Although the actual value of $\sigma$ depends on $b$, the sign of $\sigma$ does not. 

The paper is organized as follows. Existence of domain wall solutions is proved in Section 2 as a consequence of a more general existence theorem based on variational methods characterizing heteroclinics as geodesics in a degenerate metric. Section 3 is devoted to the study
of the second variation of the energy functional $E$ at the domain wall solutions. Spectral stability
follows from the properties of the second variation and is established in Section 4. 
The proof of nonlinear stability of the domain wall solutions is developed in Section 5.
Finally, Section 6 gives results on persistence and stability of the domain wall solutions in
small localized potentials by Lyapunov-Schmidt reduction analysis.

\section{Existence of heteroclinics}

In this section, the construction of the domain walls is achieved by constructing minimizers of an energy functional defined on a weak space that imposes the desired conditions at infinity and satisfies certain symmetry conditions. Later, the weak convergence is improved by looking at the second variation, which in particular implies exponential decay at infinity of $\vert U\vert^2-1,$ as well as the rest of the properties in Theorem \ref{existence}. Exponential decay is needed later to show slow motion of the center of mass of perturbations of the domain walls and to analyze stability in the presence of a small potential.
We denote the energy density
$$  e(\Psi):= \frac12\left( |\nabla\psi_1|^2 + |\nabla\psi_2|^2 + W(\psi_1,\psi_2)\right).  $$

The following theorem includes the results stated in Theorem~\ref{intro}:

\begin{theorem}
\label{existence}
Assume $W$ satisfies (W1)--(W4).  Define
\be  \label{inf} m=\inf_{\Psi\in Y}  E(\Psi).
\ee
Then there exists $\Psi=(\psi_1,\psi_2)\in Y$ which attains the infimum of $E$ in $Y$, and solves the system
\be\label{gensys}
  -\Psi''(x) + \nabla W(\Psi)=0, \qquad \lim_{x\to\infty}\Psi(x)=\mathbf{a}, \quad
\lim_{x\to-\infty}\Psi(x)=\mathbf{b}
\ee
  Moreover,
\begin{enumerate}
\item[(a)]  For any minimizing sequence $\Psi_n\in Y$, $E(\Psi_n)\to m$, there exists a minimizer $\Psi\in Y$, a sequence $\tau_n\in\RR$, and a subsequence for which
$\rho_A(\Psi_{n_k}(\cdot+\tau_{n_k}),\Psi(\cdot))\to 0$ holds for all constants $A>0$.
\item[(b)] Every minimizer has the form $\psi_1=e^{i\beta_1}u_1$,
$\psi_2=e^{i\beta_2}u_2$, for real-valued $(u_1,u_2)\in X$ and $\beta_1,\beta_2\in \RR$ constants.
\item[(c)]  All minimizers satisfy $u_1(x),u_2(x)\ge 0$.
\item[(d)]  For $W$ which obey the symmetry $W(\psi_2,\psi_1)=W(\psi_1,\psi_2)$, there exists a minimizer $U$ which is symmetric, $u_2(x)=u_1(-x)$ for all $x\in\RR$.
\item[(e)]  All minimizers exhibit exponential convergence of $|U(x)-\mathbf{a}|$ as $x\to\infty$ and $|U(x)-\mathbf{b}|$ as $x\to -\infty$.
For the system \eqref{hc}, there exist constants $C_1,C_2, R$ such that
$u_1(x)\le C_1 e^{\sqrt{\gamma - 1} x}$ for $x<-R$ and $1-u_1(x)\le C_2 e^{-\sqrt{2}x}$ for $x>R$, and similarly for $u_2(x)$.
\end{enumerate}
\end{theorem}

We note that the potential in the model case \eqref{W} satisfies the symmetry condition in (d), and hence there is a symmetric minimizing domain wall solution for equation \eqref{hc}.

We begin by establishing some basic energy estimates.

\begin{lemma}\label{limits}  Assume $W$ satisfies (W1)--(W4), and
$\Psi\in H^1_{loc}(\RR;\CC^2)$ with $E(\Psi)<\infty$.  Then
$\displaystyle\lim_{x\to\pm\infty} W(\Psi(x))=0$.
\end{lemma}

\begin{proof}  For $\Psi=(\psi_1,\psi_2)\in\CC^2$, denote
$$  |\dist\!|(\Psi,\mathbf a):= \dist\left(\, (|\psi_1|,|\psi_2|)\, , \, \mathbf a\,\right)
.  $$
Since $\mathbf a, \mathbf b$ are nondegenerate global minimizers of $W$,
there exist constants $C, \delta >0$ such that for any $\Psi\in\CC^2$ with
$|\dist\!|(\Psi,\mathbf a)\le \delta$, we have
\be\label{nondeg}
C^{-1}\sqrt{W(\Psi)}\le |\dist\!|(\Psi, \mathbf a)\le C\sqrt{W(\Psi)},
\ee
and the same estimate holding for $\mathbf b$ replacing $\mathbf a$.

Suppose $W(\Psi(x))\not\to 0$ as $x\to \pm\infty$.
Then there exists $\eps_0>0$ and a sequence $x_n\to\infty$
(or $x_n\to -\infty$) for which $W(\Psi(x_n))\ge \eps_0$ for all $n$.
By \eqref{nondeg}, we may conclude that
$$ \min\bigl\{ |\dist\!|\left( \Psi(x_n), \mathbf a\right),
|\dist\!|\left( \Psi(x_n), \mathbf b\right)\bigr\}\ge C^{-1}\sqrt{\eps_0}:=\delta_0.
$$
On the other hand, since $\int_{-\infty}^\infty W(\Psi(x))\, dx<\infty$, there also must exist sequences along which $W(\Psi(x))\to 0$.  For each $n$, let $t_n$ be the smallest value of $t>x_n$ for which either
$$  |\dist\!|(\Psi(t), \mathbf a)=\frac12\delta_0, \quad\text{or}\quad
   |\dist\!|(\Psi(t), \mathbf b)=\frac12\delta_0, \quad \forall n\in\mathbb N. $$
One of the two must occur infinitely often; suppose a subsequence may be chosen with $|\dist\!|(\Psi(t_n), \mathbf a)=\frac12\delta_0$.  By extracting further subsequences if necessary, we may assume that the two sequences $\{x_n\}$ and $\{t_n\}$ interlace:
$$   x_n<t_n<x_{n+1}, \quad \forall n\in\mathbb N.  $$

Next we observe (as in \cite{Stern}),
\be\label{MM}  \int_y^z e(\Psi) \, dx \ge \int_y^z \sqrt{W(\Psi(x))}|\Psi'(x)|\, dx
     =\int_{\sigma} \sqrt{W}\, ds,
\ee
where $\sigma$ is the path in $\RR^2$ parametrized by $\Psi(x)$, $x\in (y,z)$, and $s$ is arclength.  If we denote by $\sigma_n$ the path traced by $\Psi(x)$ for $x\in (t_n, x_n)$, we observe that
 the (Euclidean) arclength of $\sigma_n$ is at least $\delta_0/2$.
Therefore, using \eqref{nondeg} and \eqref{MM}, we may conclude that
$$   \int_{t_n}^{x_n} e(\Psi)\, dx \ge \int_{\sigma_n} \sqrt{W}\, ds
  \ge \int_{\sigma_n} C^{-1}|\dist\!|(\Psi(x),\mathbf a)\, ds \ge {\delta_0^2\over 2C}, $$
which gives a constant contribution to the total energy for each $n\in\mathbb N$.  Since the intervals $[t_n,x_n]$ are mutually disjoint, we conclude that $E(\Psi)$ diverges, a contradiction.
\end{proof}

We note that the condition (W3) may be weakened, as long as the value of $W$ controls the distance to the minima, so as to replace the condition  \eqref{nondeg}.  For instance, this would still be the case if at each of the minima $\mathbf{a}, \mathbf{b}$, $W$ vanishes to finite order.

The following useful lemma comes from \cite{ABGui}.

\begin{lemma}\label{ABG}
Let $U(x)=(u(x),v(x))\in H^1_{loc}([L_1,L_2];\RR^2)$, with
$|U(L_1)-\mathbf b|< \delta$ and $|U(L_2)-\mathbf a|<\delta$, where $\delta>0$ is as in \eqref{nondeg}.
Then, there exists a constant $C_1>0$ such that
$$   \int_{[L_1,L_2]} e(U(x))\, dx \ge m
          - C_1\left[ |U(L_1)-\mathbf b|^2 + |U(L_2)-\mathbf a|^2\right].
$$
\end{lemma}

We may now begin the proof of the existence theorem.

\begin{proof}[Proof of Theorem~\ref{existence}]  Let
$\Psi_n=(\psi_{1,n},\psi_{2,n})\in Y$ be a minimizing sequence, $E(\Psi_n)\to m$.  Let $\tau_n$ be the smallest value for which $|\psi_{1,n}(\tau_n)|=|\psi_{2,n}(\tau_n)|$. We define $\tilde\Psi_n(x):= \Psi_n(x+\tau_n)$, and note that with this definition $\tilde\Psi_n=(\tilde\psi_{1,n},\tilde\psi_{2,n})$ is a minimizing sequence for $E$ in $Y$ with normalization
\be\label{centered}|\tilde\psi_{1,n}(0)|=|\tilde\psi_{2,n}(0)|.  
\ee

By the boundedness of the energy, we may conclude that $\|\tilde\Psi'_n\|_{L^2(\RR)}$ is uniformly bounded.  Hypothesis (W4) may be integrated to obtain the estimate
$$  W(\Psi)\ge \frac12 c_0 |\Psi|^2 - c_1,  $$
for a constant $c_1$, which holds for all $\Psi\in\CC^2$.  As a consequence, for every fixed $R>0$, $\|\tilde\Psi_n\|_{H^1(-R,R)}$ is likewise uniformly bounded in $n$.
For each $R$, we may extract a subsequence $\tilde\Psi_{n_j}\to\Psi=(\psi_1,\psi_2)$ converging pointwise a.e. on $\RR$, uniformly on $[-R,R]$, and weakly in $H^1([-R,R])$.  Exhausting $\RR$ by a sequence of bounded intervals, and applying a diagonal argument, we obtain a further subsequence (which we continue to denote $\tilde\Psi_{n_j}$) for which
$\tilde\Psi'_{n_j} \weak \Psi'$ in $L^2(\RR)$ and $\tilde\Psi_{n_j}\to \Psi$ uniformly on every compact set, with $\Psi\in H^1_{loc}(\RR;\CC^2)$.  By weak lower semicontinuity of the norm and Fatou's lemma, we have
$$  E(\Psi) \le \liminf_{j\to\infty}  E(\tilde\Psi_{n_j}) = m.  $$

For the existence statement, it remains to show that $\Psi\in  Y$, and hence $\Psi$ is the desired minimizer.
 Denote by
$U_{n_j}(x)=(|\psi_{1,n_j}(x)|,|\psi_{2,n_j}(x)|)$ and
$U(x)=(|\psi_1(x)|,|\psi_2(x)|)\in X$.  We note that $E(U)\le E(\Psi)=m$, with $U_{n_j}\to U$ uniformly on compact sets and weakly in $H^1_{loc}$.

By Lemma~\ref{limits}, $\lim_{x\to\pm\infty} W(\Psi(x))=0$.
Assume, for a contradiction, that $|\Psi(x)|\to \mathbf b$ as $x\to\infty$.
For any $\eps>0$ (to be chosen later,) there exists $L_1>0$ with $\bigr|U(L_1)-\mathbf b\bigr|<\eps$.  Since $U_{n_j}\to U$ locally uniformly, there exists $j$ sufficiently large that
$|U_{n_j}(L_1)- \mathbf b|<2\eps$.  In addition, $U_{n_j}\in X$, and so there exists $L_2>L_1$ such that $|U_{n_j}(L_2)-\mathbf a|\le \eps$.  Applying Lemma~\ref{ABG},
\begin{align}\nnn
\int_{L_1}^{L_2} e(U_{n_j})\, dx &\ge m - C_1
\left[ |U_{n_j}(L_1)-\mathbf b|^2 + |U_{n_j}(L_2)-\mathbf a|^2\right]
     \\
   &\ge m- 5C_1\eps^2. \label{right}
\end{align}

On the other hand, for each $j$, $U_{n_j}\to \mathbf{b}$ as $x\to -\infty$, and $U_{n_j}$ have been normalized so that
$U_{n_j}(0)\in \Delta:=\{u\in\RR^2_+: \ u_1=u_2\}$.  Fix $\delta>0$ such that $\dist(\mathbf{a},\Delta), \dist(\mathbf{b},\Delta)>2\delta$, and
let $x_j<0$ be the largest negative value for which $|U_{n_j}(x_j)-\mathbf{b}|=\delta$.
By hypothesis (W2), there exists $w_0>0$ with
$\sqrt{W(U)}\ge w_0$ for all $U\in \RR^2_+$ with
$\dist(\mathbf{a},U), \dist(\mathbf{b},U)\ge 2\delta$.
Let $D=\dist(\Delta, B_\delta(\mathbf{b}))$.  Then for any $j$ we have:
$$ \int_{x_j}^0 e(U_{n_j}) \, dx
  \ge \int_{x_j}^0 \sqrt{W( U_{n_j}(x) )}\, |U'_{n_j}|\, dx
     = \int_{\{U_{n_j}(x): \ x_j\le x\le 0\}} \sqrt{W}\, ds
       \ge w_0 D.
$$
Together with the lower bound \eqref{right}, we then have for all sufficiently large $j$:
$$  \left[\int_{x_j}^{0} +\int_{L_1}^{L_2} \right] e(U_{n_j})\, dx
\ge m + w_0 D- 5C_1\eps^2.
$$
Taking $\eps>0$ small enough that $\eps^2<{w_0 D\over 10C_1}$ we
arrive at the contradiction
$E(\Psi_{n_j})\ge E(U_{n_j})\ge m+\frac12 w_0D$, for all sufficiently large $j$, which contradicts the definition of $\Psi_n$ as a minimizing sequence for $E$.
In conclusion, $U = (|\psi_1|,|\psi_2|)\to \mathbf a$ as $x\to +\infty$.  A similar argument shows that $U \to \mathbf b$ as $x\to -\infty$, and hence
$U\in X$ and gives the desired real-valued heteroclinic.  This completes the proof of existence of heteroclinic solutions to \eqref{gensys}.

\medskip

We now prove the properties (a)--(e) stated in Theorem~\ref{existence}.
To prove (b), let $\Psi\in Y$ be any minimizer, and $U(x)=(|\psi_1(x)|,|\psi_2(x)|)\in X$, which is 
also a minimizer of $E$.  Then, $E(U)\le E(\Psi)$
with equality if and only if $\psi_1=e^{i\beta_1}u_1$ and $\psi_2=e^{i\beta_2}u_2$,
with constant $\beta_1, \beta_2$.  Indeed, $W(\Psi)=W(|\psi_1|,|\psi_2|)$ holds for
any complex number $\Psi$.  The inequality $|(|\psi_j|)'(x)|\le |\psi'_j(x)|$ holds for almost
every $x$, with equality if and only if $\psi_j(x)=e^{i\beta_j}|\psi_j(x)|$ with
constant $\beta_j$.  Thus, all minimizing solutions must have the form specified in (b).

To prove (c), first note that by the argument of the preceding paragraph,
if $\tilde U(x)=(|u_1(x)|,|u_2(x)|)\in X$, then $E(\tilde U)\le E(U)$, with equality if and only if $u_j(x)=|u_j(x)|$, and hence all minimizers have nonnegative components.

We now consider property (d), the symmetry of minimizers, in the special case of symmetric $W$, $W(\psi_2,\psi_1)=W(\psi_1,\psi_2)$.
Let $U(x)=(|\psi_1(x)|,|\psi_2(x)|)\in X$, as above.  We note that by the choice of $\tau_n$ above, $U=(u_1,u_2)$ must satisfy $u_1(0)=u_2(0)$.
In the case
\begin{equation}\label{half}
\int_0^\infty e(U)\, dx \le \int_{-\infty}^0 e(U)\, dx,
\end{equation}
we define a new configuration $\hat U(x)$ by
$$  \hat U(x) = \begin{cases}
    ( u_1(x), u_2(x)), &\text{if $x\ge 0$,}\\
    (u_2(-x), u_1(-x)). &\text{if $x<0$.}
\end{cases}
$$
Then $\hat U\in X$, and $E(\hat U)\le E(U)=m$, so $\hat U$ is also a minimizer of $E$ in $X\subset Y$, with the desired symmetry.
In case the opposite inequality holds in \eqref{half}, we keep the values of $u_1,u_2$ for $x<0$, and perform the reflection to $x>0$ to reduce the energy.  In either case, we obtain the existence of a minimizer with symmetry as given by (d).

To prove (e) on the exponential decay of the solution as $x\to\pm\infty$, we recall the Stable and Unstable Manifold Theorem
for differential equations (see, e.g., \cite{Stable}). Both ${\bf a}$ and ${\bf b}$ are equilibrium states
of the system of differential equations \eqref{hc}, which correspond to the non-degenerate minima of $W$.
Consequently, they define saddle points of the dynamical system defined by the system of ODEs, and the linearization at the saddle points
possesses two pairs of (non-vanishing) real eigenvalues.
By the Unstable Manifold Theorem, the nonlinear dynamical system \eqref{hc} has a two-dimensional unstable manifold $W_u({\bf a})$
that is tangent to the manifold $E_u({\bf a})$ at $(u_1,u_2) = {\bf a}$.
Since the minimizer must belong to
the unstable manifold because of the boundary condition (\ref{minus}), we conclude that
the solution decays exponentially to ${\bf a}$.

For the specific equation \eqref{hc}, the linearized dynamical system has the two-dimensional unstable manifold at the point ${\bf a}$ in the explicit form
\begin{equation}
\label{manifold}
E_u({\bf a}) := \left\{ u_1' = \sqrt{\gamma -1} u_1, \quad u_2' = \sqrt{2}(u_2 - 1), \quad (u_1,u_2) \in \mathbb{R}^2 \right\}.
\end{equation}
Thanks to (c), the minimizer satisfies $u_1 > 0$ and $u_2 < 1$ in
the parametrization of $E_u({\bf a})$ in (\ref{manifold}). Thanks to (d), the result also extends to
the other infinity, where the minimizer decays exponentially to ${\bf b}$.

Finally, we turn to the convergence (a) of complex-valued minimizing sequences in the distance functions $\rho_A$.  Let $\Psi_n=(\psi_{n,1},\psi_{n,2})$ be a minimizing sequence for $E$ in $Y$.  By the first step in the existence proof, we may find translations $\tau_n$ such that $\tilde\Psi_n(x)=\Psi_n(x+\tau_n)$ is normalized with center at the origin, \eqref{centered}.  For simplicity, we assume that the original minimizing sequence $\Psi_n$ satisfies \eqref{centered}.  Denote by
$U_n=(|\psi_{n,1}|,|\psi_{n,2}|)\in X$, which we have already noted is also a minimizing sequence, with $m\le E(U_n)\le E(\Psi_n)\to m$.  
By previous arguments, there exists a minimizer $U\in X$, $E(U)=m$, and a subsequence (which we will continue to denote $U_n$) for which $U_n\to U$ pointwise a.e. on $\RR$, uniformly on compact intervals, and weakly in $H^1_{loc}$.

\medskip

\noindent {\it Step 1:} $U_n\to U$ in $L^\infty(\RR)$.

\smallskip

Suppose not:  then there exists $\eps_0>0$ and a sequence of points $x_n\to\infty$ (or $x_n\to -\infty$) such that
$$   |U_n(x_n)-U(x_n)|\ge \eps_0  $$
for all $n$.  Furthermore, $U(x_n)\to \mathbf a$ as $n\to\infty$, so there exists $N_1\in\NN$ so that
$|U(x_n)- \mathbf a|\le {\eps_0\over 10}$
for all $n\ge N_1$.  Hence,
$$  |U_n(x_n)-\mathbf a| \ge {9\eps_0\over 10} $$
holds for all $n\ge N_1$.

On the other hand, each $U_n(y)\to \mathbf a$ as $y\to\infty$, so we may choose $y_n$ to be the smallest $y>x_n$ for which
$|U(y_n)-\mathbf a|={\eps_0\over 10}$.  We thus have
\begin{equation}\label{xy}
 |U_n (y_n)-U_n(x_n)| \ge {4\eps_0\over 5}
\end{equation}
for all $n\ge N_1$.  By the estimate \eqref{nondeg},
$$  \sqrt{W(U_n(x))} \ge C^{-1} |U_n(x)-\mathbf a| \ge C^{-1}{\eps_0\over 10}, $$
for $x_n\le x\le y_n$ and $n\ge N_1$.  Applying \eqref{MM}, we have
$$
\int_{x_n}^{y_n} e(U_n)\, dx \ge \int_{\sigma_n}\sqrt{W}\, ds
   \ge C^{-1}{\eps_0\over 10}{4\eps_0\over 5}:= \eps_1,
$$
where $\sigma_n=\{U_n(x): \ x_n\le x\le y_n\}$ is a path in $\RR^2$ with (Euclidean) arclength at least ${4\eps_0\over 5}$ (by \eqref{xy}.)

Now choose $R>0$ such that $\int_{-R}^R e(U)\, dx \ge m -{\eps_1\over 10}.$  By weak lower semicontinuity, there exists $N_2\ge N_1$ such that for all $n\ge N_2$,
\begin{equation}\label{N2}   \int_{-R}^R e(U_n)\, dx
       \ge \int_{-R}^R e(U)\, dx - {\eps_1\over 10}\ge m-{\eps_1\over 5}.
\end{equation}
Therefore, for $n\ge N_2$ we have
$$  E(U_n) = \int_\RR e(U_n)\, dx \ge
\left[\int_{-R}^R + \int_{x_n}^{y_n}\right] e(U_n)\, dx
\ge m + {4\eps_1\over 5},
$$
which contradicts the fact that $U_n$ is a minimizing sequence for $E$.  Thus, Step 1 is verified.

\medskip

\noindent {\it Step 2:}  $\int_\RR W(U_n)\, dx \to \int_\RR W(U)\, dx$.

\medskip

Suppose not.  By Fatou's lemma we have $\int_\RR W(U)\le\liminf \int_\RR W(U_n)$, and so we may assume that there exists $\eps_1>0$ and a subsequence (still labelled $U_n$)
for which
$$  \int_\RR W(U_n) - \int_\RR W(U) \ge \eps_1  $$
for all $n$.  Let $R$ be chosen so that
$$
\int_{-R}^R e(U)\, dx \ge m-{\eps_1\over 10}.
$$
By uniform convergence and the arguments of Step 1, there exists $N_2\in\NN$ for which both \eqref{N2} holds and
$$
       \int_{-R}^R W(U_n)\, dx \le
       \int_{-R}^R W(U)\, dx + {\eps_1\over 5}\le \int_\RR W(U)\, dx +{\eps_1\over 5},
$$
for all $n\ge N_2$.
By the definition of $\eps_1$, it follows that either
$$  \int_R^\infty W(U_n)\, dx \ge {2\eps_1\over 5}, \quad\text{or}\quad
   \int_{-\infty}^{-R} W(U_n)\, dx \ge {2\eps_1\over 5}.
$$
Assume it is the former which holds.
But then we have the contradiction,
$$  E(U_n)=\int_\RR e(U_n)\, dx \ge
     \int_{-R}^R e(U_n)\, dx + \int_R^{\infty} W(U_n)\, dx
         \ge m + {\eps_1\over 5},
$$
for all $n\ge N_2$.
Thus Step 2 must hold.

\medskip

As a corollary to Step 2 we have:

\noindent {\it Step 3:}  $\| \Psi'_n\|_{L^2(\RR)}^2 \to \| \Psi'\|_{L^2(\RR)}^2$.

\smallskip

Indeed, as $\int_\RR W(\Psi_n) = \int_\RR W(U_n) \to \int_\RR W(U)= \int_\RR W(\Psi)$, and $E(\Psi_n)\to E(\Psi)$, it follows that
$$  \int_\RR  |\Psi'_n|^2 \, dx \to  \int_\RR |\Psi'|^2\, dx.  $$
By a familiar argument we may conclude that
$$\|\Psi'_n - \Psi'\|_{L^2(\RR)} \to 0.$$

\noindent {\it Step 4:} $U_n\to U$ in $L^2(\RR)$.

\smallskip

We first claim that for any $\eps>0$, there exists $R>0$ so that
\begin{equation}\label{uniform}
\int_{\{|x|\ge R_n\}}  W(U_n)\, dx < \eps
\end{equation}
for every $n\in\NN$.
Indeed, if we assume the contrary, then there exists a subsequence of $n\to\infty$, $\eps_0>0$ and $R_n\to\infty$ such that
$$ \int_{\{|x|\ge R\}} W(U_n)\, dx \ge \eps_0.  $$
Fix $R_0>0$ with the property that
$$  \int_{-R_0}^{R_0} W(U)\, dx \ge \int_\RR W(U)\, dx
      -{\eps_0\over 4}.  $$
By the uniform convergence $U_n\to U$ on compact intervals, there exists $N_2\in\NN$ so that when $n\ge N_2$,
$$  \int_{-R_0}^{R_0} W(U_n)\, dx \ge \int_\RR W(U)\, dx
      -{\eps_0\over 2}.  $$
It follows that
$$  \int_\RR W(U_n)\, dx \ge \left[
        \int_{-R_0}^{R_0} + \int_{\{|x|\ge R_n\}}\right] W(U_n)\, dx
           \ge \int_\RR W(U)\, dx + {\eps_0\over 2}.
$$
However, this contradicts Step 2, and thus the claim must be true.

To prove Step 4, we let $\eps>0$ be arbitrary, and recall the definition of $\delta$ from \eqref{nondeg}.
We choose $R>0$ to satisfy the following three conditions:  \eqref{uniform}, $\int_{\{|x|\ge R\}} W(U)\, dx < \eps$, and that both of
$|U_n(x)-\mathbf a|,|U(x)-\mathbf a|<\delta$ for all $x\ge R$ and for all $n$.  The first condition follows from the claim, the second from the finiteness of the integral $\int_\RR W(U)$, and the third from Step 1.
Applying \eqref{nondeg}, for all $x\in\RR$ we have:
\begin{align*}   |U_n(x)-U(x)|^2 \le
   \left(|U_n(x)-\mathbf a| + |U(x)-\mathbf a|\right)^2
     &\le C^2\left( \sqrt{W(U_n)} + \sqrt{W(U)}\right)^2 \\
      &\le 2C^2\left( W(U_n)+W(U)\right).
\end{align*}
Therefore we have
$$  \int_R^\infty |U_n(x)-U(x)|^2\, dx \le 2C^2
         \int_R^\infty \left( W(U_n)+W(U)\right) \le 4C^2\eps.
$$
A similar calculation produces the same estimate over the interval $(-\infty,R)$, and the convergence in $L^2([-R,R])$ follows from uniform convergence on compact sets.  Thus Step 4 is proven.
\medskip

Putting together the uniform convergence on compact sets $[-A,A]$ and Steps 3 and 4, we obtain the conclusion (e), $\rho_A(\Psi_n,\Psi)\to 0$ for any fixed $A>0$.
\end{proof}

The real-valued energy minimizing domain walls $U(x)$ solve the Euler-Lagrange equations,
\begin{equation}\label{EL}
-U''(x) + \nabla W(U)=0.
\end{equation}
Under the hypothesis (W4), all solutions (and not just energy minimizers) of this system are in fact {\it a priori} bounded in supremum norm:

\begin{theorem}\label{apriori}
There exists a constant $K>0$, depending only on $W$, such that for any domain wall solution $U\in X$ of \eqref{EL},
\be\label{apest}  \| U \|_{L^\infty(\RR)} \le K.  \ee
For the system \eqref{gen}, the following pointwise estimate holds:
$$   {u_1(x)^2\over a^2}+{u_2(x)^2\over b^2}\le 1.  $$
\end{theorem}

Of course, by part (b) of Theorem~\ref{existence}, the same statement may be made for any complex-valued solution $\Psi\in Y$.

\begin{proof}
From hypothesis (W4), we may easily obtain the global bound,
\be\label{gb}
\nabla W(U)\cdot U \ge c_0|U|^2 - c_1,
\ee
valid for all $U\in\RR^2$.  Define $\varphi(x):=u_1(x)^2 + u_2(x)^2-K$, for constant $K> \max\{(c_1/c_0),a^2,b^2\}$.  We calculate
\begin{align*}
 \frac12\varphi''(x) &= [u'_1]^2 + [u'_2]^2 + \nabla W(U)\cdot U \\
   &\ge c_0\left(|U|^2 - {c_1\over c_0}\right)  \\
   &\ge c_0 \varphi.
\end{align*}
Since (by the choice of $K$,) $\lim_{|x|\to\infty} \varphi(x)<0$, the positive part $\varphi_+(x)=\max\{\varphi(x),0\}$ has compact support in $\RR$.
Multiplying the equation for $\varphi$ by $\varphi_+$ and integrating, we have:
$$  \int_\RR \left[ \frac12 (\varphi'_+)^2  + c_0 \varphi_+^2\right] dx =0,
$$
and hence $\varphi(x)\le 0$ on $\RR$.  This proves \eqref{apest}.

To prove the more precise bound for solutions of \eqref{gen}, let
 $\varphi(x):= \sqrt{g_{11}}u_1^2(x)+\sqrt{g_{22}}u_2^2(x) -\mu$.  Then, $\varphi$ satisfies the equation
\begin{align*}
  -\frac12 \varphi'' + (g_{11}u_1^2+g_{22}u_2^2)\varphi &= -\sqrt{g_{11}}[u'_1]^2 - \sqrt{g_{22}}[u'_2]^2-(g_{12}-\sqrt{g_{11}g_{22}})\sqrt{g_{11}g_{22}}u_1^2u_2^2\\
  &\le 0
\end{align*}
Again, multiplying by $\varphi_+(x)=\max\{\varphi(x),0\}$ and integrating over $\RR$, we obtain:
$$  \int_\RR\left[ \frac12 (\varphi'_+)^2 + (g_{11}u_1^2+g_{22}u_2^2)\varphi_+^2\right]
\le 0,$$
so we conclude that $\varphi(x)\ge 0$ on $\RR$.  Recalling the definitions of $a,b$, we have $0\le\varphi(x)=\mu({u_1^2\over a^2}+{u_2^2\over b^2}-1)$, and the desired bound follows.
\end{proof}


%

\section{Second Variation and Monotonicity}

Looking at the second variation allows us to derive asymptotic properties of the domain walls obtained so far. In the end these properties will imply the minimizing character of these heteroclinics in a space with a stronger topology.
Let $U(x)=(u_1(x),u_2(x))\in X$ be an energy minimizing heteroclinic solution
obtained in Theorem~\ref{existence}.  For real-valued $U$, the second variation of energy, $D^2 E(U)$ may be expressed using the definition $W(\Psi)=F(|\psi_1|^2,|\psi_2|^2)$ in the following form,
\begin{align*}  D^2 E(U)[\Phi] &:=
    \left.{d^2\over d\eps^2}\right|_{\eps=0} E(U+\eps\Phi) \\
    &=
   \int_{-\infty}^\infty
\bigl\{ |\nabla\Phi|^2
+
\partial_1F(u_1^2,u_2^2) |\varphi_1|^2
  + \partial_2F(u_1^2,u_2^2) |\varphi_2|^2
  +  2\bigl[ \partial_1^2 F(u_1^2,u_2^2) (u_1,\varphi_1)^2  \\
  &\qquad
     + 2\partial_1\partial_2 F(u_1^2,u_2^2) (u_1,\varphi_1)(u_2,\varphi_2)
      + \partial_2^2 F(u_1^2,u_2^2) (u_2,\varphi_2)^2 \bigr]\bigr\}
dx,
\end{align*}
%
where $\Phi=(\varphi_1,\varphi_2)\in H^1_0(\RR;\CC^2)$ and the inner product of
complex numbers $(z,w):= \Re(\bar z\, w)$.
Having chosen $U=(u_1,u_2)\in X$, from Theorem~\ref{existence} we may
conclude that $D^2 E(U)$ is well-defined for $\Phi\in C_0^\infty(\RR;\CC^2)$, and
in fact we may extend its domain to include any $\Phi\in H^1(\RR;\CC^2)$.

Let $U=(u_1,u_2)$ be a real-valued minimizer of $E$ in the class $X.$
Writing $\Phi=\Phi_R + i\Phi_I$, with $\Phi_R=(\varphi_{1,R},\varphi_{2,R})$,
$\Phi_I=(\varphi_{1,I},\varphi_{2,I})$ in its real and imaginary parts,
we associate to the quadratic form $D^2 E(U)$ the two self-adjoint linearizations $L_+$ and $L_-$ so that
\begin{equation}
\label{second-variation}
D^2 E(U)[\Phi] = \langle \Phi_R, L_+\Phi_R\rangle +
                                   \langle \Phi_I, L_-\Phi_I\rangle.
\end{equation}
With this decomposition, we obtain formulae for $L_\pm$.  First, the real part is given in terms of the usual real-valued linearization of the Euler-Lagrange equations,
\begin{align}\nnn
L_+\Phi_R &= \begin{bmatrix}
-\partial_x^2 + \partial_1F(u_1^2,u_2^2) + 2\partial_1^2F(u_1^2,u_2^2) u_1^2
   & 2\partial_1\partial_2F(u_1^2,u_2^2) u_1u_2  \\
2\partial_1\partial_2F(u_1^2,u_2^2) u_1u_2 &
-\partial_x^2 + \partial_2F(u_1^2,u_2^2) + 2\partial_2^2F(u_1^2,u_2^2) u_2^2
\end{bmatrix}\Phi_R \\
\label{L+}
&= \partial_x^2\Phi_R + \frac12 D^2W(U)\Phi_R.
\end{align}
The imaginary part is a diagonal operator:
\begin{align}\nnn
L_-\Phi_I &= \begin{bmatrix}
-\partial_x^2 + \partial_1F(u_1^2,u_2^2) & 0  \\
0 &
-\partial_x^2 + \partial_2F(u_1^2,u_2^2) 
\end{bmatrix}\Phi_I \\
\label{L-}
&= \partial_x^2\Phi_I + DF(u_1^2,u_2^2):\Phi_I.
\end{align}
%
%

Properties of the second variation are characterized in the following theorem. They record useful information that will allow us to derive spectral stability of the linearized operator about $U.$

\begin{theorem}\label{secondvar}  Assume (W1)--(W4), and let $U=(u_1,u_2)$ by any energy minimizing solution of \eqref{gensys} in $X$.
\begin{enumerate}
\item[(i)]
The quadratic form $D^2 E(U)$ is positive semi-definite on $H^1_0(\RR;\CC^2)$, and each operator $L_\pm$ is also positive semi-definite on $H^1_0(\RR;\RR^2)$.
\item[(ii)]  Zero is an eigenvalue of $L_+$, with associated eigenfunction $U'(x)$.
\item[(iii)]  $\sigma_{ess}(L_-)=[0,\infty)$, and there exists $\Sigma_0>0$ with $\sigma_{ess}(L_+)=\left[\Sigma_0,\infty\right)$.
\end{enumerate}
If in addition we assume,
$$
\partial_1\partial_2F(\xi)\gneq 0 \quad\text{for all $\xi\in\RR^2_+$,}
\leqno{{\rm (W5)}}
$$
\noindent we have:
\begin{enumerate}
\item[(iv)]  $u'_1(x)>0$ and $u'_2(x)<0$ for all $x\in\RR$.
\item[(v)]  Zero is a simple eigenvalue of $L_+$.
\end{enumerate}
\end{theorem}

An easy calculation shows that we obtain the full results of the theorem in the special cases given in the Introduction:
\begin{corollary}\label{2varcor}
All conclusions (i)--(v) are valid for solutions $U(x)\in X$ of \eqref{gen} and of \eqref{ho}.
\end{corollary}

\begin{proof}[Proof of Theorem~\ref{secondvar}]
By the first part of the proof of Theorem~\ref{existence}, $U\in X\subset Y$ also minimizes $E$ over the space $Y$, and hence for any $\Phi\in C_0^\infty(\RR;\CC^2)$, we must have $D^2E(U)[\Phi]\ge 0$.  By density we may then conclude that the quadratic form $D^2E(U)\ge 0$ on $H^1(\RR;\CC^2)$.  By restricting to real-valued $\Phi$, we have $\langle \Phi, L_+\Phi\rangle \ge 0$ for all $\Phi\in H^1(\RR;\RR^2)$, and thus the self-adjoint operator $L_+\ge 0$.  Similarly, we obtain $L_-\ge 0$ by considering only variations $\Phi= i\Phi_I$ with $\Phi_I \in H^1(\RR;\RR^2)$, and thus (i) is verified.

\medskip

By (e) of Theorem~\ref{existence} we may conclude that $U'(x)=(u'_1(x),u'_2(x))\in H^1(\RR;\RR^2)$.  Moreover, by direct calculation we see that $L_+ U' =0$ holds for $x\in\RR$.  Thus, $\lambda=0$ is an eigenvalue of $L_+$, and $\lambda_0:=\inf\sigma(L_+)=0$, and (ii) is true in the general case.

Statement (iii) now follows from the asymptotic behavior of $U$ at infinity.
By property (e) of Theorem \ref{existence}, the decay of $U$ to either
${\bf a}$ or ${\bf b}$ is exponential. By Weyl's Lemma, see, e.g., \cite{HS},
the essential spectrum of $L_{\pm}$ coincide with the union of the spectra
of constant-coefficient operators
\begin{align*}
L_+^- =  -\partial_x^2 + \frac12D^2W(\mathbf{b}), \quad
L_+^+ = -\partial_x^2 + \frac12D^2W(\mathbf{a})
\end{align*}
and
\begin{align*}
L_-^- =   -\partial_x^2 + DF(\mathbf{b}^2):\  , \quad
L_-^+ = -\partial_x^2 + DF(\mathbf{a}^2):\  ,
\end{align*}
where $L_{\pm}^- = \lim_{x \to -\infty} L_{\pm}$ and $L_{\pm}^+ = \lim_{x \to +\infty} L_{\pm}$.
Since $L_{\pm}^{\pm}$ are constant-coefficient operators, their spectra are continuous, with lower bound given by the smallest eigenvalue of the (constant) potential matrix.  For $L_+^\pm$, we recall from (W3) that $\mathbf{b},\mathbf{a}$ are nondegenerate minima of $W$, and hence we may choose $\Sigma_0>0$ to be the smallest eigenvalue among those of $D^2W(\mathbf{b}),D^2W(\mathbf{a})$, and conclude that $\sigma_{ess}(L_+)=[\Sigma_0,\infty)$.

For $L_-^\pm$, we note that $\mathbf{a}$ being a critical point of $W$, we must have
$\partial_1F(\mathbf{a}^2)=\partial_1F(a^2,0)=0$, and hence $\sigma(L_-^+)=[0,\infty)$.  The same argument applies to $L_-^-$, and so we may conclude that (iii) holds in the general case.

Next, assume that $\partial_1\partial_2F(\xi)\ge 0$ for all $\xi\in\RR^2_+$.
To prove (iv), we use the variational characterization of $\lambda_0=\inf_{\|\Phi\|_{L^2}=1} \langle \Phi, L_+\Phi\rangle$.  As $\lambda_0$ is an eigenvalue, the infimum is attained by some $\Phi=(\varphi_1,\varphi_2)\in H^1(\RR;\RR^2)$.  We observe that
$\tilde\Phi=(|\varphi_1|,-|\varphi_2|)\in H^1(\RR;\RR^2)$, and
by (W5),
$\langle \tilde\Phi, L_+\tilde\Phi\rangle \le \langle \Phi, L_+\Phi\rangle$, and hence for any null vector $\Phi$, $\tilde\Phi$ is also an eigenvector.   Note that $\tilde\varphi_j$ solve equations of the form
$$  \tilde\varphi_1'' -c_{11}(x)\tilde\varphi_1 = c_{12}(x)\tilde\varphi_2 \le 0, \qquad
   \tilde\varphi_2'' + c_{22}(x)\tilde\varphi_2 = c_{21}(x)\tilde\varphi_1 \ge 0,
   $$
with coefficient matrix $[c_{ij}(x)]_{i,j}=-\frac12 D^2W(U)$ having all negative entries.  Thus, by the strong maximum principle applied to each equation individually, we may conclude that 
$
\tilde\varphi_1(x)>0$ and  $\tilde\varphi_2(x)<0$ for all $x\in\RR$.  As a consequence,
$\tilde\Phi=\Phi$, and all eigenfunctions corresponding to $\lambda=0$ satisfy 
\be\label{mon}
\tilde\varphi_1(x)>0, \quad \tilde\varphi_2(x)<0, \quad\text{for all $x\in\RR$}.
\ee
  Since $U'$ is such an eigenfunction, (iv) must hold.

The simplicity of the ground-state eigenvalue $\lambda_0=0$ now follows from a standard argument.  Indeed, assume that $\dim\ker(L_+)\ge 2$.  Then, there exists an eigenfunction $\Phi\in H^1(\RR;\RR^2)$, $L_+\Phi=0$, which is orthogonal to $U'$, $\langle U',\Phi\rangle =0$.  By the above paragraph, $\Phi=(\varphi_1,\varphi_2)$ must also satisfy \eqref{mon}, whence $0=\langle U',\Phi\rangle =\int_\RR [u'_1\varphi_1 + u'_2\varphi_2]dx>0$, a contradiction.  Thus (v) is proven.

\end{proof}

We observe that, at least in the more concrete example \eqref{hc}, the positivity of the second variation follows directly from the Euler-Lagrange equations themselves, without reference to energy minimization.  Indeed,  For the operator $L_+$, we write
$\varphi_{1,R} := A_1 u_1'$ and $\varphi_{2,R} := A_2 u_2'$ with $x$-dependent $A_{1,2}$, where
the components $u_{1,2}$ satisfy the differential equations (\ref{hc}). Integrating by parts,
we obtain
\begin{eqnarray*}
\langle \Phi_R, L_+\Phi_R\rangle & = & \int_{-\infty}^{\infty}
\left[ (A_1')^2 (u_1')^2 + (A_2')^2 (u_2')^2 + A_1^2 [(3 u_1^2 + \gamma u_2^2 - 1) (u_1')^2 - u_1' u_1''' ] \right. \\
& \phantom{t} & \left. + A_2^2 [(\gamma u_1^2 + 3 u_2^2 - 1) (u_2')^2 - u_2' u_2''' ] + 4 \gamma A_1 A_2 u_1 u_2 u_1' u_2' \right] dx.
\end{eqnarray*}
Substituting derivatives of the system (\ref{hc}), we obtain
\begin{eqnarray*}
\langle \Phi_R, L_+\Phi_R\rangle = \int_{-\infty}^{\infty}
\left[ (A_1')^2 (u_1')^2 + (A_2')^2 (u_2')^2 - 2 \gamma u_1 u_2 u_1' u_2' (A_1-A_2)^2 \right] dx.
\end{eqnarray*}
Since $u_{1,2} > 0$, $u_1' > 0$, and $u_2' < 0$, we confirm that the quadratic form
is non-negative and touches zero at only one eigenvector that corresponds to
$x$-independent $A_1$ and $A_2$ satisfying the constraint $A_1 = A_2$.

For the operator $L_-$, we write
$\varphi_{1,I} := B_1 u_1$ and $\varphi_{2,I} := B_2 u_2$ with $x$-dependent $B_{1,2}$. Integrating by parts
and using the differential equations (\ref{hc}), we obtain
\begin{eqnarray*}
\langle \Phi_I, L_- \Phi_I\rangle & = & \int_{-\infty}^{\infty}
\left[ (B_1')^2 u_1^2 + (B_2')^2 u_2^2 + B_1^2 [(u_1^2 + \gamma u_2^2 - 1) u_1^2 - u_1 u_1'' ] \right. \\
& \phantom{t} & \left. + B_2^2 [(\gamma u_1^2 + u_2^2 - 1) u_2^2 - u_2 u_2'' ] \right] dx \\
& = & \int_{-\infty}^{\infty}
\left[ (B_1')^2 u_1^2 + (B_2')^2 u_2^2 \right] dx.
\end{eqnarray*}
These computations shows that the quadratic form
is non-negative. 

\section{Spectral Stability}

Spectral stability of the domain wall solutions follows from analysis of eigenvalues in the linear eigenvalue problem
associated with the perturbation $(\Phi_R + i\Phi_I)e^{i t \lambda}$ of
the domain wall solutions $U$.
Here $U$ is a real-valued minimizer of $E$ in function space $X$
and $\Phi_R=(\varphi_{1,R},\varphi_{2,R})$, $\Phi_I=(\varphi_{1,I},\varphi_{2,I})$ are components
of the eigenvector in ${\rm Dom}(L_{\pm}) \subset L^2(\mathbb{R})$ that correspond to the eigenvalue
$\lambda \in \mathbb{C}$ of the associated spectral problem. 

\vspace{0.25cm}

\begin{proof1}{\em of Theorem \ref{theorem-stability}.}

We know from Theorem~\ref{secondvar} that the spectrum of $L_{+}$ has a zero eigenvalue of finite multiplicity (moreover this zero eigenvalue is simple whenever $W$ satisfies (W5)), while the rest of the spectrum is bounded from below by a positive number.
Considering $V:=(\ker L_{+})^{\perp},$ the orthogonal complement of the nullspace of $L_{+}$ in $L^2(\mathbb{R}),$ one has that for any nonzero eigenvalue $\lambda$ of the spectral stability problem (\ref{eigenvalue}),
the component $\Phi_I$ must belong to ${\rm Dom}(L_-) \cap V$.

We proceed to show any nonzero eigenvalue $\lambda$ must be purely imaginary.
To that end let $P$ denote the orthogonal projection from $L^2(\mathbb{R})$ to $V.$
Let $\lambda\neq 0.$ Since $\Phi_I=P\Phi_I$ and $\lambda$ is an eigenvalue associated to (\ref{eigenvalue}), we can decompose $\Phi_R$ as
$$
\Phi_R = - \lambda P L_+^{-1} P \Phi_I + (\mathbf{1}-P)\Phi_R,
$$
where $\mathbf{1}$ is the identity operator in $L^2(\mathbb{R}).$ Furthermore $ (\mathbf{1}-P)\Phi_R$ can be expressed uniquely as
$$
 (\mathbf{1}-P)\Phi_R=\sum_{i=1}^n c_i f_i,
$$
for some coefficients $c_i,$ where $f_1,\ldots, f_n$ form an orthonormal basis of $\ker L_{+}.$
The coefficients $c_1,\ldots,c_n$ can be equivalently found from $\Phi_I$ as
$$
c_i=\frac{\langle f_i,L_{-}\Phi_I\rangle}{\lambda}.
$$
The linear eigenvalue problem (\ref{eigenvalue}) for $\lambda\neq 0$ is equivalent to the generalized eigenvalue problem
\begin{equation}
\label{generalized}
P L_- P \Phi_I = -\lambda^2 P L_+^{-1} P \Phi_I,
\end{equation}
As in \cite[p.468]{Gallo}, the characterization of the eigenvalues of the generalized eigenvalue problem \eqref{generalized} is given by the Rayleigh quotient:
\begin{equation}
\label{Rayleigh}
-\lambda^2 = \inf_{\Phi \in {\rm Dom}(L_-^{\infty}) \cap V, \Phi \neq 0}
\frac{\langle L_- \Phi, \Phi \rangle}{\langle L_+^{-1} \Phi, \Phi \rangle}.
\end{equation}
By Theorem~\ref{secondvar},
there exists $C > 0$ such that $\langle L_+^{-1} \Phi, \Phi \rangle \geq C \| \Phi \|^2$ for all $\Phi\in{\rm Dom}(L_-) \cap V(\mathbb{R})$,
whereas $\langle L_- \Phi, \Phi \rangle \geq 0$. Therefore, $-\lambda^2 \geq 0$, hence $\lambda \in i \mathbb{R}$.
This proves that the domain wall solutions are spectrally stable for any choice of paramters $g_{ij},\mu$ in \eqref{gen}, \eqref{genlim}.
\end{proof1}

\begin{remark}
Note that under the assumption that zero is a simple eigenvalue of $L_{+}$ then $V$ takes the simple form
$$
V:=L^2_c(\mathbb{R}) := \left\{ \Phi \in L^2(\mathbb{R}) : \langle U', \Phi \rangle = 0 \right\}.
$$
In this case also, the decomposition of $\Phi_R$ is simply given by
$$
\Phi_R = - \lambda P L_+^{-1} P \Phi_I + a U',
$$
this time $a$ can be computed from $\Phi_I$ as
$$
a = \frac{\langle U', L_- \Phi_I \rangle}{\lambda \langle U', U' \rangle}.
$$
\end{remark}

\section{Nonlinear Stability}

In this section we prove the orbital stability of the domain walls of \eqref{GP} found as local minimizers of $E$ in $Y.$  We have almost all the elements in place; global well posedness, conservation of energy of which the domain walls are minimizers. One thing is missing; if we are given a minimizing sequence in $Y,$ we do not know if its limit coincides with $U.$ We conjecture that the minimizer $U(x)$ of the energy (\ref{energy}) in function space $X$
is unique, up to translation and gauge invariance.
For our purposes it is enough to show that, should there be several real-valued minimizers
$U$ of $E$ in $X$, then each one is isolated in the $H^1(\RR;\RR^2)$ norm (modulo translation.)

\begin{prop}\label{isolated} Let $U=(u_1,u_2)\in X$ be any energy minimizing solution of \eqref{hc}.  Assume that $\lambda=0$ is an isolated, simple eigenvalue of $L_+$ (defined as in \eqref{second-variation}.)
Then there exists $\eta_0>0$ such that if $V=(v_1,v_2)\in X$ is any other solution of \eqref{hc}, then either
$$  \inf_{\tau\in\RR} \| U(\cdot)-V(\cdot -\tau)\|_{L^2}\ge \eta_0, $$
or there exists $\tau\in \RR$ such that $V(x-\tau)=U(x)$.
\end{prop}

We note that by Theorem~\ref{secondvar} the hypothesis on the ground state of $L_+$ is satisfied for the Gross-Pitaevskii examples \eqref{GP}, \eqref{gen}.

\begin{proof}
First, fix a solution $U\in X$ of \eqref{hc}.  We observe that
\begin{equation}\label{UUp}  \int_\RR U\cdot U'(x)\, dx = \int_\RR \frac12 {d\over dx} |U(x)|^2\, dx
    = 0.
\end{equation}
Let $V\in X$ be any solution of \eqref{hc} which is geometrically distinct from $U$; that is, $V(x-\tau)\neq U(x)$ holds for every $\tau\in\RR$.  We first claim that there exists $\sigma\in\RR$ which attains the minimum value of
$$  \inf_{\tau\in\RR} \| U(\cdot-\tau) - V(\cdot)\|_{L^2(\RR)} =
          \| U(\cdot-\sigma) - V(\cdot)\|_{L^2(\RR)}.
$$
Indeed, let $f(\tau):= \int_\RR \left( U(x-\tau) - V(x)\right)^2\, dx.$  Then, $f$ is differentiable on $\RR$, and
$\lim_{\tau\to\pm\infty} f(\tau)=+\infty$.  Thus, the minimum value of $f$ is achieved at some $\sigma\in\RR$.  Furthermore, $\sigma$ is a critical point, and hence
\begin{align*}
0=f'(\sigma) &= -2\int_\RR \left[ U(x-\sigma)- V(x)\right] \cdot U'(x-\sigma)\, dx \\
&= 2 \int_\RR  V(x)\cdot U'(x-\sigma)\, dx,
\end{align*}
using \eqref{UUp}.  Thus, we have the additional orthogonality condition:
\begin{equation}\label{Vs}
    \int_\RR V(x+\sigma)\, U'(x)\, dx =0.
\end{equation}
Denote by $V^\sigma(x)=V(x+\sigma)$, with $\sigma$ as in \eqref{Vs}.  Let $\Phi=V^\sigma - U\in H^1(\RR;\RR^2)$ by the decay estimate (c) of Theorem~\ref{existence}.  Note that by \eqref{UUp}, \eqref{Vs}, we have
$\langle \Phi , U'\rangle_{L^2} =0$, i.e., $\Phi\in Z:=\text{span}\,\{U'\}^\perp$.

Finally, we prove the Proposition by contradiction: suppose $V_n$ is a sequence of solutions of \eqref{hc} with $E(V_n)=m$, and
$\inf_{\tau\in\RR}\| V_n(\cdot+\tau)-U(\cdot)\|_{L^2(\RR)}\to 0$.
Let $\sigma_n\in\RR$ be chosen as above, so that \eqref{Vs} holds for each $V_n$, and let $V_n^{\sigma_n}(x):=V(x+\sigma_n)$.  Define
$\Phi_n:= V_n^{\sigma_n}- U$.  We write the equation satisfied by $V^{\sigma_n}_n$ in the form $0 = \mathcal{G}(V^{\sigma_n}_n)= -\partial_x^2 V^{\sigma_n}_n + DW(V^{\sigma_n}_n)$, and use the Taylor expansion to second order on the function $G$:  for each $n$ there exists $s_n\in (0,1)$ such that
$$   0=G(V^{\sigma_n}_n)=G(U+\Phi_n)=
           G(U) + DG(U)\Phi_n + \frac12 D^2G(U+s_n\Phi_n)[\Phi_n,\Phi_n].  $$
Since $U$ solves the equation $G(U)=0$.  Moreover, $DG(U)\Phi_n=L_+\Phi_n$, and $D^2G(U)=D^3W(U)$, and thus the above equation yields:
$$     L_+\Phi_n = -\frac12 D^3W(U+s_n\Phi_n)[\Phi_n,\Phi_n].
$$
Set $\tilde \Phi_n:= \Phi_n/\|\Phi_n\|_{L^2}\in Z$.  Then, by homogeneity we have
\begin{equation}\label{Phi_eq}    L_+\Phi_n =-\frac12 \|\Phi_n\|_{L^2}\, D^3W(U+s_n\Phi_n)[\tilde\Phi_n,\tilde\Phi_n].
\end{equation}
By \eqref{apest} (see Theorem~\ref{apriori}), there is a universal constant $C_1$ such that $\|U(x)\|_{L^\infty}, \|V_n(x)\|_{L^\infty}\le C_1$, and hence $\|\Phi_n\|_\infty=\|V_n-U\|_\infty\le 2C_1$.  By (W1), $D^3W$ is uniformly bounded on bounded sets, and thus we may estimate:
\begin{align} \label{Lest}  \langle \tilde\Phi_n, L_+\tilde\Phi_n\rangle_{L^2} &=
         -\frac12  \int_\RR \sum_{i,j,k}
                \partial_{ijk}W(U+s_n\Phi_n)\tilde\varphi_{n,i}\tilde\varphi_{n,j}\varphi_{n,k} \\
                \nnn
         &\le C_1 \|\Phi_n\|_{L^\infty}\|\tilde\Phi_n\|_{L^2}^2 \le C_2,
\end{align}
is uniformly bounded in $n$.  From this we conclude the uniform bound on the derivatives,
$$  \int_\RR  |\tilde\Phi'_n|^2 dx \le \langle \tilde\Phi_n, L_+\tilde\Phi_n\rangle_{L^2} - \int_\RR D^2W(U)[\tilde\Phi_n,\tilde\Phi_n]\, dx
\le C_2 + C_1\|\tilde\Phi_n\|_{L^2}^2=C_2+C_1.
$$
By the Sobolev embedding we conclude that $\|\tilde\Phi_n\|_{L^\infty}\le C_3$ is uniformly bounded, and we may improve the estimate \eqref{Lest},
\begin{align*}
  \langle \tilde\Phi_n, L_+\tilde\Phi_n\rangle_{L^2}
  &= -\frac12 \|\Phi_n\|_{L^2} \int_\RR \sum_{i,j,k}
                \partial_{ijk}W(U+s_n\Phi_n)\tilde\varphi_{n,i}\tilde\varphi_{n,j}\tilde\varphi_{n,k}  \\
    &\le C_1\|\Phi_n\|_{L^2}\|\tilde\Phi_n\|_{L^2}^2\|\tilde\Phi_n\|_{L^\infty} \to 0.
\end{align*}
Since $\tilde\Phi_n\in Z$ and $\|\tilde\Phi_n\|_{L^2}=1$ for all $n$, we arrive at a contradiction, as the quadratic form $\langle \Phi, L_+\Phi\rangle_{L^2}$ is strictly positive definite for $\Phi\in Z$.  In conclusion, each real minimizer $U$ is isolated in $L^2(\RR)$ norm.
\end{proof}

By combining Proposition~\ref{isolated} with statement (b) of Theorem~\ref{existence}, we have:

\begin{corollary}\label{complex}
Under the hypotheses of Proposition~\ref{isolated}, for any complex-valued minimizer $\Psi\in Y$ of $E(\Psi)$, either
$$  \inf_{\tau,\alpha_1,\alpha_2\in\RR}
\left\| U(\cdot)-\begin{bmatrix} e^{i\alpha_1}\psi_1\\ e^{i\alpha_2}\psi_2\end{bmatrix}(\cdot -\tau)\right\|_{L^2}\ge \eta_0, $$
or there exists $\tau,\alpha_1,\alpha_2\in \RR$ such that $e^{i\alpha_j}\psi_j(x-\tau)=u_j(x)$, $j=1,2$.
\end{corollary}

A careful inspection of the proof of Proposition \ref{isolated} leads us to a further corollary,
where we recall the definition of the energy space $\mathcal{D}$ from \eqref{D}.

\begin{corollary}\label{Sch}
Let $U=(u_1,u_2)$ be a minimizing solution of \eqref{hc} and assume
zero is a simple eigenvalue of $L_+$.
Then, there are constants $l_1, l_2, \varepsilon_0>0$  such that if $\Xi\in\mathcal{D}$ with
\begin{eqnarray}\label{above}
&& l_1<\inf_{\theta_1,\theta_2,\alpha\in\mathbb{R}}\rho_A\left(\Xi,(\exp i\theta_1 u_1(\cdot+\alpha),\exp i\theta_2 u_2(\cdot+\alpha))\right)<l_2,\nonumber\\
&& \mbox{then }\nonumber\\ && E(\Xi)>\inf_{Y} E(\cdot)+\varepsilon_0.
\end{eqnarray}
\end{corollary}

Now, we turn to the proofs of Theorems \ref{orb-stab} and \ref{Slowmotion}.

\vspace{0.25cm}

\begin{proof1}{\em of Theorem \ref{orb-stab}.}  We reason by contradiction and assume, given $A,\varepsilon>0,$  that for a sequence of positive numbers $\delta_n\to 0,$ we can find $\Psi_0^n\in \mathcal{D}\cap L^{\infty}$ and times $t_n$ such that, denoting by $\Psi^n(x,t)$ the unique global in time solution to \eqref{GP} with initial condition $\Psi_0^n,$ we have $\rho_A(\Psi_0^n, U)\leq \delta_n$  and for any $\theta_1, \theta_2 \in \mathbb{R}$

\begin{equation}\label{contr} \inf_{\alpha \in \mathbb{R}}\rho_A(\Psi^n(\cdot, t_n), (\exp i \theta_1 u_1(\cdot+\alpha), \exp i \theta_2 u_2 (\cdot +\alpha))\geq \varepsilon. \end{equation}

Now, because $\rho_A(\Psi_0^n, U)\to 0,$ Fatou's lemma implies that $E(\Psi_0^n)\to m=\inf_{Y} E(\cdot) .$  From conservation of energy along the flow we deduce the same for $E(\Psi^n(\cdot, t_n)).$  The $\Psi^n (\cdot, t_n)$'s are not necessarily elements of $Y,$ but we modify them so as to obtain a genuine minimizing sequence whose limit can be compared to a member of the orbit of  $U.$

Let $R_n$ be a sequence of positive numbers such that $R_n>A$ and
\begin{equation}\label{hp}
\int_{-R_n}^{R_n} e(\Psi^n(\cdot, t_n))
>E(\Psi^n(\cdot,t_n))-\frac{1}{n}.
\end{equation}
Define:
\begin{equation}
\hat{\psi}_n(x):=\left\{
\begin{matrix}
\Psi^n(x,t_n), \quad x\in[-R_n,R_n]\\
(f_n(x),g_n(x)),\quad x\not\in[-R_n,R_n]
\end{matrix}
\right.
\end{equation}
where $(f_n(x),g_n(x))$ is a continuous vector function satisfying
$$
(f_n(\pm R_n),g_n(\pm R_n))=\Psi^n(\pm R_n,t_n),
$$
$$
\lim_{x\to\pm\infty}(f_n(x),g_n(x))=\left(\frac{1\pm 1}{2},\frac{1\mp 1}{2}\right)
$$
and such that
$$
E(\hat{\psi}_n)<E(\Psi^n(\cdot,t_n))+\frac{1}{n}.
$$
Note that by \eqref{hp} and definition of $\hat{\psi}_n,$ $\rho_A(\hat{\psi}_n,\Psi^n(\cdot,t_n))\to 0.$

We appeal to part (a) of Theorem \ref{existence} to conclude that there is $\tau_n\in\mathbb{R}$ such that $\hat{\psi}_n(\cdot+\tau_n)$ converges in the topology induced by $\rho_A$ to a minimizer $V$ of $E$ in $Y.$

From Theorem \ref{existence} part (a) we also know $V=(e^{i\beta_1}v_1,e^{i\beta_2}v_2)$ where $(v_1,v_2)$ is a minimizing solution of \eqref{hc}.
By \eqref{contr}, Fatou's Lemma and Proposition \ref{isolated} we deduce
$$
\inf_{\tau\in\mathbb{R}}\Vert(v_1,v_2)-U(\cdot+\tau)\Vert_{L^2}\geq\eta_0.
$$
By continuity of the flow, the above implies that for some $\tilde{t}^n>0$
$$
l_1<\inf_{\theta_1,\theta_2,\alpha\in\mathbb{R}}\rho_A(\Psi^n(\cdot,\tilde{t}^n),(\exp i\theta_1 u_1(\cdot+\alpha),\exp i\theta_2 u_2(\cdot+\alpha)))<l_2,
$$
which by Corollary \ref{Sch} yields $E(\Psi^n(\cdot,\tilde{t}^n))=E(\Psi^n(\cdot,t_n))=E(\Psi_0^n)>m+\varepsilon_0,$
for $n$ large enough, a contradiction.
    \end{proof1}

As mentioned in the introduction, an easy adaptation  of the analysis in \cite{Bethuel} yields Theorem \ref{Slowmotion}.
We only highlight the main steps for convenience of the reader.

\vspace{0.25cm}

\begin{proof1}{\em of Theorem \ref{Slowmotion}.}
As in \cite{Bethuel}, it suffices to verify the following claim:
\begin{claim}\label{claim}
Given any $\eps > 0$ and $A > 0$, there exists some constant $K$, depending only on $A$, and some positive number $\delta > 0$ such that, if $U$ and $\Psi_0$ are as in Theorem~\ref{orb-stab} and if  \eqref{delta-nonlinear} holds, then
$$  |a(t)| \le K\eps,  $$
for any $t \in [0, 1]$, and for any of the points $a(t)$ satisfying inequality \eqref{epsilon-nonlinear} for some
$\theta(t) \in\RR$.
\end{claim}
Indeed, once the claim is established, looking at the flow for $t\in [n,n+1 ) ,$ appealing to Claim \ref{claim} and using induction on $n ,$ the arguments of \cite{Bethuel} may be repeated verbatim, as the conclusion of Theorem~\ref{Slowmotion} follows without regard to the specific equation, depending only on energy conservation and the well-posedness of the initial value problem for the evolution equation.

 To adapt the proof in \cite{Bethuel} in our case we need to identify the following key elements.

{\it Center of mass.}  As a first step, we need an integral expression involving $U(\cdot-\alpha (t)),$ depending on quantities controlled along the flow by the energy, that can serve to follow $\alpha (t).$

{\it Approximating $\alpha(t)$ in terms of $\psi$. }  In a second step, the same integral expression above applied to $\psi$ needs to provide a well approximation of the one with $U(\cdot -\alpha (t))$  as a consequence  of orbital stability.

{\it Motion identity.}  We need a measure of how these integral expressions change; they should be controlled by a quantity that can be made arbitrarily small, again by pure energy considerations, uniformly in $t\in [0,1].$ This is the third step needed for the right setup.

According to this, to prove the claim, we choose a function which identifies a center of mass, in the spirit of \cite{Bethuel}.  

{\bf Step $1.$}

Let $U$ be a domain wall solution of \eqref{gen}, and choose a translation $U(\cdot-\tau)$ with the property that
\be\label{normal}  \int_\RR  x( \sqrt{g_{11}} u_1^2(x-\tau) + \sqrt{g_{22}} u_2^2(x-\tau) -\mu)\, dx =0.
\ee
Clearly such a choice is possible, as the integral above may be made arbitrarily large by choosing a large positive $\tau$, and arbitrarily negatively large for large negative $\tau$.  Denote by $U(x)$ the domain wall solution normalized in this way.  (For the symmetric case \eqref{hc}, $U$ is the symmetric solution.)  Define
$$  m(U):=\int_\RR ( \sqrt{g_{11}} u_1^2(x) + \sqrt{g_{22}} u_2^2(x) -\mu)\, dx.  $$
We note that for this $U$ (normalized as in \eqref{normal}) and any $a\in\RR$, we have
$$  {1\over m(U)} \int_\RR  x( \sqrt{g_{11}} u_1^2(x-a) + \sqrt{g_{22}} u_2^2(x-a) -\mu)\, dx = a.  $$

{\bf Step $2.$} 

Let $g:\RR\to [0,1]$ be a smooth function with $g(x)=1$ for $|x|\le 1$ and $g(x)=0$ for $|x|\ge 2$, and set $g_R(x)=g(x/R)$, $R\ge 1$.
For $\Psi\in Y$ and $a\in \RR$ define
$$  G_{a,R}(\Psi):= {1\over m(U)}\int_\RR g_R(x) x
    ( \sqrt{g_{11}} |\psi_1|^2(x-a) + \sqrt{g_{22}} |\psi_2|^2(x-a) -\mu)\, dx .
$$

We note now that without loss of generality one can assume $\alpha(0)=0$ by taking $\delta$ sufficiently small in \eqref{delta-nonlinear}. The exponential convergence of $U$ to its limits at $\pm \infty$ imply that
$$
G_{a,R}(U)\to G(U), \mbox{ as }R\to\infty.
$$
Furthermore, this convergence is uniform in $a$ on compact sets.
Because of this, we fix $R_0>0$ such that
\begin{equation}\label{alphaepsilon}
\vert G_{\alpha(t),R_0}(U)-\alpha(t)\vert<\varepsilon,
\mbox{ for all }t\mbox{ such that }\vert\alpha(t)\vert<1.
\end{equation}
Next, we note that Cauchy-Schwartz and  \eqref{epsilon-nonlinear} yield
\begin{equation}\label{alphapsi}
\vert G_{\alpha(t),R_0}(U)-G_{0,R_0}(\psi)\vert\leq C\varepsilon,
\end{equation}
for some positive constant $C$ independent of $\varepsilon.$
The above implies that $G_{0,R_0}(\psi)$ can be used to follow $\alpha(t)$
 with a precision of $\mathcal{O}(\varepsilon)$ which is what we wanted. We then turn to the quantitative dynamical property that lets us to exploit the approximation of $\alpha(t)$ by $G_{0,R_0}(\psi).$

\vskip.1in
{\bf Step $3.$}

 In the third step  one controls the evolution of the center of mass by considering a localized version of its motion law in terms of the momentum. As in Proposition~4.1 of \cite{Bethuel}, we have for any $R>0,$
\begin{multline}\label{centerofmass}
\frac{d}{dt}\int_\RR g_R(x)x ( \sqrt{g_{11}} |\psi_1|^2(x) + \sqrt{g_{22}} |\psi_2|^2(x) -\mu)\, dx \\
=2\int_{\mathbb{R}}[\sqrt{g_{11}}\langle i\psi_1,\partial_x \psi_1\rangle
+\sqrt{g_{22}}\langle i\psi_2,\partial_x\psi_2 \rangle]\partial_x(x\; g_R)dx.
\end{multline}
Indeed, differentiating under the integral sign, appealing to \eqref{GP}, and using the fact that for $j=1,2:$
$$
\int\langle i\psi_j,\psi_j(\vert\psi_j\vert^2+\gamma\vert\psi_{3-j}\vert^2-1)
\rangle(xg_R)=0,
$$
\eqref{centerofmass} follows.

To finish, let $\theta_1, \theta_2$ be such that $\vert\psi_1\vert =e^{i\theta_1}\psi_1$ and $\vert\psi_2\vert =e^{i\theta_2}\psi_2.$
Because
$$\int[\sqrt{g_{11}}\langle i e^{i\theta_1}u_1,\partial_x  (e^{i\theta_1}u_1)\rangle
+\sqrt{g_{22}}\langle i e^{i\theta_2} u_2,\partial_x (e^{i \theta_2}u_2) \rangle]\partial_x(x\; g_R)dx =0,$$
and again Cauchy-Schwartz together with \eqref{epsilon-nonlinear}, we see that
\begin{equation}\label{momentumslow}
\left\vert\int[\sqrt{g_{11}}\langle\psi_1,\partial_x\psi_1\rangle+\sqrt{g_{22}}\langle\psi_2,\partial_x\psi_2\rangle]\partial_x(x\;g_R)dx\right\vert\leq C\varepsilon.
\end{equation}

Let $\psi_{0,1}, \psi_{0,2}$ be the components of the initial condition $\psi_0.$

Thanks to \eqref{momentumslow}
\begin{eqnarray}
&&\left\vert\int_\mathbb{R} g_{R_0}(x)x \( \sqrt{g_{11}} (|\psi_1|^2-\vert\psi_{0,1}\vert^2) + \sqrt{g_{22}} (|\psi_2|^2-\vert\psi_{0,2}\vert^2)\)\, dx \right\vert\nonumber\\
&&=2\left\vert\int_0^t\int_{\mathbb{R}}[\sqrt{g_{11}}\langle i\psi_1,\partial_x \psi_1\rangle
+\sqrt{g_{22}}\langle i\psi_2,\partial_x\psi_2 \rangle]\partial_x(x\; g_R)dx\right\vert\nonumber\\
&&\leq C\varepsilon.
\end{eqnarray}
This last inequality together with \eqref{alphaepsilon} and \eqref{alphapsi} concludes the proof of the claim provided it holds that $\vert\alpha(t)\vert<\frac{1}{2}$ for all $t\in[0,1].$
This is a consequence of the continuity of the flow with respect to $\rho_A$ as in \cite{Bethuel}.

\end{proof1}

\section{Domain walls under a small localized potential}

We shall now consider persistence and stability of the domain wall solutions in the presence of a small
localized potential. Domain walls in this case satisfy the system of differential equations (\ref{hc-potential}).
We provide details for the symmetric case \eqref{hc}, but the same phenomena may be produced for more general equations, such as \eqref{gen} with only minor modifications.  We note that the linearized operators around the symmetric domain wall solution $U(x)\in X$ in this case may be represented as follows:
\begin{equation}
\label{operator-L-plus}
L_+ \Phi_{R} := \begin{pmatrix}  -\partial_x^2 + 3u_1^2 +\gamma u_2^2-1 & 2 \gamma u_1 u_2 \\
2 \gamma u_1 u_2 & -\partial_x^2 + \gamma u_1^2 + 3u_2^2-1 \end{pmatrix}
\begin{pmatrix} \varphi_{1,R} \\ \varphi_{2,R} \end{pmatrix},
\end{equation}
and
\begin{equation}
\label{operator-L-minus}
L_- \Phi_{I} := \begin{pmatrix}  -\partial_x^2 + u_1^2 +\gamma u_2^2-1 & 0  \\
0 & -\partial_x^2 + \gamma u_1^2 + u_2^2-1 \end{pmatrix}
\begin{pmatrix} \varphi_{1,I} \\ \varphi_{2,I} \end{pmatrix}.
\end{equation}

\vspace{0.25cm}

\begin{proof1}{\em of Theorem \ref{theorem-persistence}.}
Let us write
\begin{equation}
\label{decompositon}
U(x+s) = U_0(x) + W(x),
\end{equation}
where the leading term $U_0 = (u_1,u_2) \in X$ is the heteroclinic solution of the system (\ref{hc})
in Theorem \ref{existence} satisfying the symmetry reduction
$$
u_2(x) = u_1(-x) \quad \mbox{\rm for all} \quad x \in \mathbb{R},
$$
the parameter $s \in \mathbb{R}$ is to be uniquely determined from the condition
$\langle U_0', W \rangle = 0$, and the perturbation term $W = (w_1,w_2)$ satisfies the perturbed system
\begin{equation}
\label{inhomog-w}
L_+ W = - \epsilon V(x+s) (U_0 + W),
\end{equation}
where $L_+$ is given by (\ref{operator-L-plus}). By Theorem \ref{secondvar},
zero is the simplest and smallest eigenvalue of $L_+$ with ${\rm Ker}(L_+) = {\rm span}\{ U_0' \}$,
whereas the rest of the spectrum of $L_+$ is bounded away from zero.
To invert $L_+$ on the right-hand side of (\ref{inhomog-w}), we add the bifurcation equation
\begin{equation}
\label{bifurcation-eq}
F(s) := -\epsilon \langle U_0', V(x+s) (U_0 + W) \rangle = 0.
\end{equation}
If $F(s) = 0$ and $V \in L^2(\mathbb{R})$, the inhomogeneous equation (\ref{inhomog-w}) is closed
for $W \in H^2(\mathbb{R})$ and the implicit function theorem can be applied for sufficiently small $\epsilon$.
As a result, there exists a unique solution of (\ref{inhomog-w})
for $W \in H^2(\mathbb{R})$ subject to the orthogonality condition $\langle U_0', W \rangle = 0$
such that $\| W \|_{H^2(\RR)} \leq C |\epsilon|$ for some $C > 0$. Because the nonlinearity is polynomial,
the solution $W$ is a smooth ($C^{\infty}$) function of $\epsilon$.

Using this solution for $W$ in the bifurcation equation (\ref{bifurcation-eq}) and integrating by parts, we obtain
$$
F(s) = \frac{1}{2} \epsilon \int_{\RR} V'(x+s) ( u_1^2 + u_2^2 - 1) dx + \mathcal{O}(\epsilon^2) = 0.
$$
Since $V \in C^2(\RR)$ and conditions (\ref{condition-1}) and (\ref{condition-2}) are assumed,
the implicit function theorem for scalar functions yields that there exists a unique
solution of the bifurcation equation (\ref{bifurcation-eq}) near $x_0$ such that
$|s - x_0| \leq C |\epsilon|$ for some $C > 0$. This construction completes the proof of the theorem.
Bound (\ref{smallnest}) follows by the triangle inequality and the Sobolev embedding of $H^2(\RR)$ to $L^{\infty}(\RR)$.
\end{proof1}

To consider stability of persistent domain wall solutions in the small localized potential,
we need a technical result that ensures that property (b) of Theorem \ref{existence} persists for
small values of $\epsilon$.

\begin{lemma}
\label{lemma-positivity}
In addition to the conditions of Theorem \ref{theorem-persistence},
assume that $V \in L^1(\mathbb{R})$. Then, the  heteroclinic solutions in Theorem \ref{theorem-persistence}
satisfy $0 \leq u_1(x), u_2(x) \leq 1$ for all $x \in \RR$.
\end{lemma}

\begin{proof}
Since $\gamma > 1$, the decay of the unperturbed domain wall solution
$U_0 = (u_1,u_2)$ of the system (\ref{hc}) to the equilibrium states
${\bf a}$ and ${\bf b}$ is exponential with the decay rates
$$
u_1(x) \sim e^{\sqrt{\gamma-1}x}, \quad 1-u_2(x) \sim e^{\sqrt{2}x}, \quad \mbox{\rm as} \quad x \to -\infty
$$
(see property (d) in Theorem \ref{existence}). If the perturbation term $W = (w_1,w_2)$
in the decomposition (\ref{decompositon}) also decays exponentially to zero with the same decay rate,
the assertion of the lemma follows from the smallness of $W$ in the bound (\ref{smallnest}) and
the property (b) of Theorem \ref{existence} for the unperturbed solution.
However, since $V \in L^1(\mathbb{R})$, the exponential decay of $W$ to zero with the decay rates
$$
w_1(x) \sim e^{\sqrt{\gamma-1}x}, \quad w_2(x) \sim e^{\sqrt{2}x}, \quad \mbox{\rm as} \quad x \to -\infty
$$
follows from Levinson's Theorem for differential equations (Proposition 8.1 in \cite{CL}).
\end{proof}

\vspace{0.25cm}

\begin{proof1}{\rm of Theorem \ref{theorem-stability-potential}.}
Spectral stability is considered again in the framework of the linear eigenvalue problem
\begin{equation}
L_+(\epsilon) \Phi_R = -\lambda \Phi_I, \quad L_-(\epsilon) \Phi_I = \lambda \Phi_R,
\label{eigenvalue-perturbed}
\end{equation}
where $L_{\pm}(\epsilon)$ include the perturbed domain wall solution $U = (u_1,u_2)$ as
well as the small potential $\epsilon V$. In particular, $L_{\pm}(\epsilon)$ are given by
\begin{equation}
\label{operator-L-plus-potential}
L_+(\epsilon) := \begin{pmatrix}  -\partial_x^2 + \epsilon V + 3u_1^2 +\gamma u_2^2-1 & 2 \gamma u_1 u_2 \\
2 \gamma u_1 u_2 & -\partial_x^2 + \epsilon V + \gamma u_1^2 + 3u_2^2-1 \end{pmatrix}
\end{equation}
and
\begin{equation}
\label{operator-L-minus-potential}
L_-(\epsilon) := \begin{pmatrix}  -\partial_x^2 + \epsilon V + u_1^2 +\gamma u_2^2-1 & 0 \\
0 & -\partial_x^2 + \epsilon V + \gamma u_1^2 + u_2^2-1 \end{pmatrix}
\end{equation}
These operators admit the power expansion $L_{\pm}(\epsilon) = L_{\pm}(0) + \epsilon L_{\pm}'(0) + \mathcal{O}(\epsilon^2)$
thanks to the smoothness of $U$ in $\epsilon$ in Theorem \ref{theorem-persistence}.

We first show that for small values of $\epsilon$, the operator $L_+(\epsilon)$ is strictly positive
and bounded away from zero if $\sigma > 0$ and has exactly one negative eigenvalue with the rest of spectrum bounded
away from zero if $\sigma < 0$. Since $0$ is the simplest and smallest eigenvalue of $L_+$, the result follows
from the perturbation expansions. In particular, let us define solutions of the linear inhomogeneous
equations
\begin{equation}\label{perturbation-equations}
\left.\begin{aligned}
- w''_1(x) + \left( 3 u_1^2 + \gamma u_2^2 -1\right) w_1 + 2 \gamma u_1 u_2 w_2 &= -V u_1, \\
- w''_2(x) + \left( \gamma u_1^2 + 3 u_2^2 -1\right) w_2 + 2 \gamma u_1 u_2 w_1 &= -V u_2.
\end{aligned}\right\} \quad \quad x\in\RR,
\end{equation}
where $(u_1,u_2)$ is the unperturbed domain wall solution of the system (\ref{hc}). Then, we have
$$
L_+'(0) = \begin{pmatrix}  V + 6 u_1 w_1 + 2 \gamma u_2 w_2 & 2 \gamma u_1 w_2 + 2 \gamma u_2 w_1 \\
2 \gamma u_1 w_2 + 2 \gamma u_2 w_1 & V + 2 \gamma u_1 w_1 + 6 u_2 w_2 \end{pmatrix}
$$
The isolated zero eigenvalue of $L_+(0)$ becomes positive (negative) eigenvalue of $L_+(\epsilon)$ for small
values of $\epsilon$ if $\sigma > 0$ ($\sigma < 0$), where
\begin{eqnarray*}
\sigma & = & \langle U', L_+'(0) U' \rangle \\
& = & \int_{\mathbb{R}} \left( V (u_1')^2 + V (u_2')^2 + 6 u_1 w_1 (u_1')^2 + 6 u_2 w_2 (u_2')^2 \right) dx \\
& \phantom{t} &
+ \int_{\RR} \left( 2 \gamma u_2 w_2 (u_1')^2 + 2 \gamma u_1 w_1 (u_2')^2 + 4 \gamma u_1 w_2 u_1' u_2' + 4 \gamma u_2 w_1 u_1' u_2' \right) dx
\end{eqnarray*}
Differentiating the inhomogeneous system (\ref{perturbation-equations}) in $x$ and
projecting it to $U'$, we reduce the previous expression for $\sigma$ to the form
\begin{eqnarray*}
\sigma = - \int_{\RR} V'( u_1 u_1' + u_2 u_2') dx = \frac{1}{2} \int_{\RR} V''(u_1^2 + u_2^2 - 1) dx,
\end{eqnarray*}
where integration by parts has been performed for $V \in C^2(\mathbb{R})$. Thus, the assertion on
the spectrum of $L_+(\epsilon)$ is proven.

Next, the spectrum of $L_-(\epsilon)$ is not affected for any small $\epsilon$ compared to
the statement of Theorem \ref{secondvar} because $L_-(\epsilon)$ is a diagonal composition
of Schr\"{o}dinger operators and $L_-(\epsilon) \Phi_{1,2} = 0$ with $\Phi_1 = (u_1,0)$ and
$\Phi_2 = (0, u_2)$, where $u_{1,2}$ are positive according to Lemma \ref{lemma-positivity}.
As a result, $\sigma(L_-(\epsilon)) = [0,\infty)$ as follows from the Sturm's Theorem.

As in the proof of Theorem \ref{theorem-stability}, we construct the generalized
eigenvalue problem
\begin{equation}
\label{generalized-potential}
L_-(\epsilon) \Phi_I = -\lambda^2 L_+^{-1}(\epsilon) \Phi_I,
\end{equation}
where $L_+^{-1}(\epsilon)$ exists without any projection operators for
any $\epsilon \neq 0$. If $\sigma > 0$, then $L_+^{-1}(\epsilon)$ is
strictly positive implying
$$
-\lambda^2 = \inf_{\Phi \in {\rm Dom}(L_-(\epsilon)), \Phi \neq 0}
\frac{\langle L_-(\epsilon) \Phi, \Phi \rangle}{\langle L_+^{-1}(\epsilon) \Phi, \Phi \rangle} \geq 0.
$$
as in \cite[p.468]{Gallo}. This yields stability of the heteroclinic solutions.
If $\sigma < 0$, $L_+^{-1}(\epsilon)$ has exactly one negative eigenvalue. As in Theorem 3.1 in \cite{Board},
this condition implies that there exists exactly one negative eigenvalue $-\lambda^2$
of the generalized eigenvalue problem (\ref{generalized-potential}). This yields instability of the heteroclinic solutions.
\end{proof1}

\end{document}